\documentclass[11pt]{amsart}
\usepackage{amsmath}
\usepackage{amssymb}
\usepackage{amscd, mathrsfs}
\usepackage{amsthm}
\usepackage[hidelinks]{hyperref}

\usepackage{tikz}
\usetikzlibrary{cd}
\usetikzlibrary{decorations.markings}

\usepackage{amsfonts}
\usepackage{subfigure}
\usepackage{comment}

\usepackage[all]{xy}
\usepackage{color}

\usepackage{enumerate} %for enumitem[(i)]
\usepackage{mathtools} %coloneqq
\usepackage{microtype}
\usepackage{cancel}
\numberwithin{equation}{section}

\newtheorem{thm}[equation]{Theorem}
\newtheorem{proposition}[equation]{Proposition}
\newtheorem{prop}[equation]{Proposition}

\newtheorem{lemma}[equation]{Lemma}

\newtheorem{corollary}[equation]{Corollary}
\newtheorem{cor}[equation]{Corollary}

\theoremstyle{definition}
\newtheorem{definition}[equation]{Definition}

\newtheorem{defn}[equation]{Definition}

\newtheorem{example}[equation]{Example}
\newtheorem{remark}[equation]{Remark}

\def\XXint#1#2#3{{\setbox0=\hbox{$#1{#2#3}{\int}$} 
		\vcenter{\hbox{$#2#3$}}\kern-.5\wd0}}

\newcommand{\N}{\mathbb N}

\newcommand{\Cl}{\mathrm {Cl}}

\newcommand{\R}{\mathbb R}

\renewcommand{\S}{{\mathcal S}}

\renewcommand{\ker}{\operatorname{Ker}}

\newcommand{\Ad}{\operatorname{Ad}}
\newcommand{\ad}{\operatorname{ad}}

\newcommand{\Center}{\mathcal{Z}}
\newcommand{\Ideal}{\mathfrak{I}}

\newcommand{\sus}{\subseteq}
\newcommand{\g}{\mathfrak{g}}
\newcommand{\s}{\mathfrak{s}}
\newcommand{\h}{\mathfrak{h}}

\newcommand{\inn}{\in \{1, \ldots , n\}}
\newcommand{\edge}[1]{\mathfrak{e}_{#1}} 
\newcommand{\pedge}[1]{\mathfrak{e}(#1)}  % \edge with paranthesis
\newcommand{\lie}[1]{\mathrm{Lie}(#1)}

\renewcommand{\ker}{\operatorname{Ker}}

\newcommand{\spann}[1]{\mathrm{span}\{#1\}}

\newcommand{\doo}[1]{\partial{}_{#1}} 

\tikzset{->-/.style={decoration={
			markings,
			mark=at position #1 with {\arrow{>}}},postaction={decorate}}}
\tikzset{-<-/.style={decoration={
			markings,
			mark=at position #1 with {\arrow{<}}},postaction={decorate}}}

\newcommand{\ja}{\quad \text{and} \quad}

\usepackage{bbold}

\begin{document}
%%%%Uncomment to hide the proofs
%\excludecomment{proof}
	
	\title[Semigenerated Lie algebras]{
		Semigenerated step-3 Carnot algebras and applications to sub-Riemannian perimeter}
	%On the possible regularities of sets with constant horizontal normal in Carnot groups}
	
	%  
	%
	\author{Enrico Le Donne}
	\author{Terhi Moisala}

	\address{\textsc{Enrico Le Donne}: 
		Dipartimento di Matematica, Universit\`a di Pisa, Largo B. Pontecorvo 5, 56127 Pisa, Italy \\
		\& \\
		University of Jyv\"askyl\"a, Department of Mathematics and Statistics, P.O. Box 35 (MaD), FI-40014, Finland}
	\email{enrico.ledonne@unipi.it}
	
	\address{\textsc{Terhi Moisala}:
		University of Jyv\"askyl\"a, Department of Mathematics and Statistics,
		P.O. Box 35 (MaD), FI-40014, Finland}
	\email{moisala.terhi@gmail.com}

	%\keywords{Path isometry; embedding; Sub-Riemannian manifold; Nash Embedding Theorem; Lipschitz embedding.}

	\renewcommand{\subjclassname}{%
		\textup{2010} Mathematics Subject Classification}
	%\subjclass[]{30L05, 53C17, 26A16}

	\date{\today}

	\renewcommand{\subjclassname}{%
		\textup{2010} Mathematics Subject Classification}
	
	\thanks{
		% C. B. was partially supported by the EPSRC grant EP/S005641/1 and by the National Science Foundation Grant No. DMS-1638352.
		E.L.D. was partially supported by the Academy of Finland (grant
		288501
		`\emph{Geometry of subRiemannian groups}' and by grant
		322898
		`\emph{Sub-Riemannian Geometry via Metric-geometry and Lie-group Theory}')
		and by the European Research Council
		(ERC Starting Grant 713998 GeoMeG `\emph{Geometry of Metric Groups}').
	}
	\begin{abstract}
%		This paper contributes to the study of sets of finite intrinsic perimeter in Carnot groups. Our intent is to characterize in which groups the only sets with constant intrinsic normal are the vertical half-spaces. Our view point is algebraic: such a "feature" happens if and only if the semigroups generated by each horizontal half-space is a vertical half-space.
%		For Carnot groups of nilpotency step 3 we provide a complete characterization of those  groups for which this happens in terms of whether such groups don't have any Engel-type quotients. Such latter groups are introduced here, are the minimal (in terms of quotients) counterexsemigenerateds, and have an interesting geometry per se.
%		In addition, for Carnot groups of arbitrary step, we give some sufficient criteria to conclude that sets with constant intrinsic normal are vertical half-spaces. For doing this, we define a new class of Carnot groups, which we call type $\Diamond$ and generalizes the previous type $\star$ defined by M. Marchi.
 
	This paper contributes to the study of sets of finite intrinsic perimeter in Carnot groups. Our intent is to characterize in which groups the only sets with constant intrinsic normal are the vertical half-spaces. Our viewpoint is algebraic: such a phenomenon happens if and only if the semigroup generated by each horizontal half-space is a vertical half-space. We call \emph{semigenerated} those Carnot groups with this property.
	For Carnot groups of nilpotency step 3 we provide a complete characterization of semigeneration in terms of whether such groups do not have any Engel-type quotients. Engel-type groups, which are introduced here, are the minimal (in terms of quotients) counterexamples.
	In addition, we give some sufficient criteria for semigeneration of Carnot groups of arbitrary step. For doing this, we define a new class of Carnot groups, which we call type $(\Diamond)$ and which generalizes the previous notion of type $(\star)$ defined by M. Marchi. As an application, we get that in type $ (\Diamond) $ groups and in step 3 groups that do not have any Engel-type algebra as a quotient, one achieves a strong rectifiability result for sets of finite perimeter in the sense of Franchi, Serapioni, and Serra-Cassano.
	\end{abstract}
	
	\maketitle
	
	%\addcontentsline{toc}{section}{Contents}
	\setcounter{tocdepth}{2}
	\tableofcontents

\section{Introduction}
	 Carnot groups, which are by definition simply connected Lie groups with stratified Lie algebras, raised attention because of their natural occurrences in Geometric Measure Theory and Metric Geometry. 
%	In Geometric Measure Theory and Metric Geometry, Carnot groups, which are by definition simply connected Lie groups with stratified Lie algebras, raised attention because of their natural occurrences with atypical geometry due to their possibly non-abelian algebraic structure.
	In particular, subsets of Carnot groups whose intrinsic normal is constantly equal to a left-invariant vector field appear both in the development of a theory \`a la De Giorgi for sets of locally finite perimeter in sub-Riemannian spaces \cite{fssc2, fssc,MR3353698} 
	%Ambrosio_Gezzi_Magnani
	and in the obstruction results for bi-Lipschitz embeddings into $L^1$ of non-abelian nilpotent groups \cite{Cheeger-Kleiner2}.
	The work \cite{fssc} by Franchi, Serapioni and Serra-Cassano provides complete understanding of sets with constant intrinsic normal in the case of Carnot groups with nilpotency step 2 by proving that they are half-spaces when read in exponential coordinates. However, in higher step the study appears to be much more challenging due to the more complex underlying algebraic structure, and only in the case of type $ (\star) $ groups and of filiform groups we have a satisfactory understanding of sets with constant intrinsic normal, see \cite{Marchi,Bellettini-LeDonne}.
	
	In a recent paper \cite{Bellettini-LeDonne2}, C. Bellettini and the first-named author of this article related the property of having constant intrinsic normal to the containment of distinguished constant-normal sets, which are semigroups generated by the horizontal half-space defined by the normal, as we shall explain soon. We shall use the following terminology: a \emph{horizontal half-space} of a stratified algebra $ \g $ with horizontal layer $ V_1 $ is the closure of either of the two parts into which a hyperplane divides $ V_1 $. A \emph{vertical half-space} is defined as the direct sum of a horizontal half-space and the derived subalgebra $ [\g,\g] $. By \cite[Corollary 2.31]{Bellettini-LeDonne2}, in exponential coordinates a Carnot group has the property that all its constant-normal sets are equivalent to vertical half-spaces if and only if the closure of the semigroup generated by 
	%	the exponential of 
	each horizontal half-space is 
	%	the exponential of 
	a vertical half-space.
	In Carnot groups with this property, one has the intrinsic $ C^1 $-rectifiability result for finite-perimeter sets 
	\`a la De Giorgi. %, Franchi, Serapioni and Serra-Cassano.
	In arbitrary groups, the study of semigroups can still give some weaker rectifiability results, see
	\cite{DLMV}.
	
	In this paper, we continue the study of such semigroups from an algebraic viewpoint. In particular, we get to a complete characterization of those step-3 Carnot groups for which all constant-normal sets are vertical half-spaces. In addition, for Carnot groups of arbitrary nilpotency step, we give some sufficient criteria which generalize the previous work by M. Marchi \cite{Marchi}.
	
%	Before stating the results, we introduce some notation and definitions. Let $G$ be a nilpotent simply connected Lie group with Lie algebra $\g$ and diffeomorphic exponential map $\exp:\g\to G$. 
%	For every subset  $W\subseteq \g$ 
%	we shall denote by $S_W $ the semigroup generated by $\exp(W)$ in $G$.
% and by 
%	$\mathfrak s_W\subseteq \g$ the set such that
%	$$\exp(\mathfrak s_W)=S_W.$$
	
%	
%	If $\g$ is a stratified Lie algebra that is stratified as $\g=V_1 \oplus
%	\cdots \oplus  V_s$ we say that 
%	$W$ is a 
%	{\em horizontal half-space} 
%	of $\g$ if 
%	$W$ is a closed half-space in $V_1$.
	% there exists a non-zero element $\lambda$ in the dual of $V_1$ such that
	%\begin{equation}
	%\label{halfspace}
	%W= \lambda^{-1}( [ 0, +\infty)) \subseteq V_1.\end{equation}
	% and in this case we define
	% $$\partial W:= \lambda^{-1}( \{ 0\}),$$ 
	
	\begin{definition} Given a Carnot group $ G $ with exponential map $\exp:\g\to G$,
	we say that a set  $W\subseteq \g$ is  {\em   semigenerating} if the closure of the  semigroup generated by $\exp(W)$ in $G$ contains the commutator subgroup $[G,G]$. We say also that the Lie algebra $\g$ is  {\em   semigenerated} if every horizontal half-space $W$ in $\g$ is semigenerating. 
	\end{definition}
\noindent
	We shall use the term \emph{Carnot algebra} to denote the (stratified) Lie algebra of a Carnot group, which is completely determined by the Lie group, see \cite{LeDonne:Carnot}.
	%	equals to $ \exp (W  \oplus V_2\oplus	\ldots V_s)$.
	By the work \cite{fssc} of Franchi, Serapioni and Serra-Cassano, we know that step-2 Carnot algebras are semigenerated. Their work has been then extended by Marchi to a class of Carnot algebras, called of type $(\star)$, which includes examples of arbitrarily large nilpotency step. However, the basic example given by the Engel Lie algebra is not semigenerated, see  \cite{fssc,Bellettini-LeDonne}, and also Proposition \ref{prop:Engels:non-CNP}. From this example, it is easy to generate more examples of non-semigenerated algebras, because of the observation that each quotient of a semigenerated Lie algebra is semigenerated, see Proposition~\ref{prop:quotients_product_ample}. 
	Thus, for example  we have that 
	no stratified Lie algebra of rank 2 and step $\geq 3$ is   semigenerated because each of them has the Engel Lie algebra as quotient as pointed out in Remark~\ref{rmk:filiforms_engel_quots}.
	
	 Here, we mostly focus on step-3 Lie algebras, in which we discover a class of Lie algebras that are not semigenerated. Since they are a generalization of the Engel Lie algebra we call them Engel-type algebras. Our main result is that these algebras are the only obstruction to semigeneration.
	
	\begin{thm}\label{t:nonCNP:char:step6}
		Let $ \g $ be a stratified Lie algebra of step at most 3. Then $ \g $ is not   semigenerated if and only if it has one of the Engel-type algebras (as in Definition~\ref{def:nthEngel:intro}) as a quotient.
	\end{thm}
	
	\begin{defn}
		\label{def:nthEngel:intro}
		For each $ n\in \N $, we call \emph{$ n $-th Engel-type algebra} the $ 2(n+1) $-dimensional Lie algebra (of step 3 and rank $ n+1 $) with basis $ \{X,Y_i,T_i,Z\}_{i=1}^n $\,, where the only non-trivial bracket relations are given by $ [Y_i,X]=T_i $ and $ [Y_i,T_i] = Z $ for all $i \inn $.
	\end{defn}
	
	It is a challenge to understand how one can express in pure combinatorial terms the property of not having any Engel-type algebra as quotient. However, we have examples of step-3 Lie algebras that are not of type $(\star)$ but have no   Engel-type quotients. Hence, our result is a strict generalization of \cite{Marchi}.
	It is possible that Theorem~\ref{t:nonCNP:char:step6} holds also in case the nilpotency step is arbitrary; we have no counterexample. However, the situation in step greater than 3 is more technical. For this reason, we are only able to give a sufficient condition to ensure semigeneration in arbitrary step.
	Such criterion is not necessary (see Example~\ref{ex:tame_not_necessary}); however,
	as for the type $(\star)$ condition, it is computable in terms of brackets of some particular basis.
	
	In the next result, we assume the existence of a basis with specific properties. We could restate the condition in other forms (see Lemma~\ref{lemma:diamond_TFAE}), which alas are just as technical.
	
	\begin{definition}\label{def:diamond} 	Let $\g$ be a stratified Lie algebra. If, for each subalgebra $ \h $ of $ \g $ for which 
	$ \h\cap V_1 $ has codimension 1 in $V_1$,   there exists a basis $ \{X_1,\ldots,X_m\} $ of $ V_1 $ such that
		\begin{equation*}%\label{eq:diamond_basis}
		\ad_{X_i}^2 X_j \in\h \ja \ad_{\ad_{X_i}^kX_j}^2 (X_i) \in \h,\qquad \text{ for all } i,j = 1,\ldots, m \text{ and } k\geq 2 ,
		\end{equation*}
%		for all $ i,j = 1,\ldots, m $ and $ k\geq 2 $,
	then we say that $\g$ is  {\em of type $(\Diamond)$}.
\end{definition}
	
	\begin{thm}
		\label{thm:diamond}
		Every stratified Lie algebra 
		that  is of type $(\Diamond)$ (as in Definition~\ref{def:diamond})  is semigenerated.  
	\end{thm}

	To put the results in perspective, we remind the reader that by \cite{fssc}, we know that if a Carnot group has the property that every set with constant intrinsic normal is a vertical half-space, then every set of locally finite sub-Riemannian perimeter have a strong rectifiability property. Since semigroups generated by horizontal half-spaces are minimal constant-normal sets with respect to set inclusion according to \cite{Bellettini-LeDonne2}, we obtain the following corollary.
	\begin{corollary}
	If the Lie algebra of a Carnot group is semigenerated (e.g., if it is of  type $ (\Diamond) $, see Theorem~\ref{thm:diamond}, or has step 3 and does not have any Engel-type algebra as a quotient, see Theorem~\ref{t:nonCNP:char:step6}), then the reduced boundary of every set of locally finite perimeter in $ G $ is intrinsically $ C^1 $-rectifiable.
%		If $ G $ is a Carnot group whose Lie algebra is of type $ (\Diamond) $ or that has step at most 3 and does not have an Engel-type algebra as a quotient, then the reduced boundary of every set of locally finite perimeter in $ G $ is intrinsically $ C^1 $-rectifiable.
	\end{corollary}
	
	The structure of the article is the following.
In 	Section~\ref{sec:prelim} we discuss some preliminaries. In addition to the notions of semigenerated and trimmed algebras, we introduce a useful set called the edge of a semigroup.	In Section~\ref{sec:sufficient} we analyze Lie algebras of type $(\Diamond)$ and prove  Theorem~\ref{thm:diamond}, see Corollary~\ref{cor:dimaond_ample}.
	Section~\ref{sec:low} is devoted to both a list	of examples	and of results valid for Carnot algebras of step at most 4.
In	Section~\ref{sec:Engel} we study the Engel-type algebras. 
%{\color{violet}	We prove that they are the only trimmed non-ample Carnot algebras with step 3. Consequently, we end with the proof of Theorem~\ref{t:nonCNP:char:step6}. }
	We show that they are the only non-semigenerated Carnot algebras with step 3 that are minimal with respect to quotient in a sense that will be made precise with a notion that we call trimmed (see Definition \ref{minimal}). We end with the proof of Theorem~\ref{t:nonCNP:char:step6}.

\section{Preliminaries}	\label{sec:prelim}
	We start with a small list of notations.  Then, in Definition~\ref{def:wedge}, we define the edge $\edge{\s}$ and the wedge $\mathfrak w _\s$ of a semigroup $\s$.  The notion of edge will be in the core of the discussion, since understanding if a horizontal half-space is semigenerating reduces to calculating the edge of its generated semigroup.  We provide several preliminary results regarding the size of such edges. In particular, we consider Lemma~\ref{lemma:new_invariant_directions} extremely useful and we shall exploit it repeatedly.
	In Proposition~\ref{lemma:minimal_TFAE} we provide equivalent conditions for the definition of trimmed algebra, a notion that is fundamental in our arguments in Section~\ref{sec:Engel}. 
%	Also Lemma~\ref{subalg} will be used few times.

	In this paper, the Lie algebra  $ \g $ will always be  stratified with layers $V_i=V_i(\g)$.
	We denote   by $\Center(\g)$ the center of a Lie algebra $\g$.
	Given an ideal   $ \mathfrak i $ of   $ \g $ we denote by
	$\pi=
	\pi_{\mathfrak i}:\g \to \g/\mathfrak i $ the quotient map, and we shall interchangeably use  the equivalent notations
	$ A/\mathfrak i =A+ \mathfrak i= \pi_{\mathfrak i}(A) ,$ for subsets $A$ of $\g$.
	We denote for a subset $ A $ of $ \g $ by 
	$\Ideal_\g(A)$ the ideal generated by $ A $ within $\g$, by $ \lie{A} $ the Lie algebra generated by $A$, and by  $\Cl(A)$ or $\bar A$ the closure of $ A $ in $\g$.

	%If   $\g$ is a stratified Lie algebra that is stratified as $\g=V_1 \oplus
	%\ldots V_s$ we
	We say that 
	$W$ is a 
	{\em horizontal half-space} 
	of $\g$ if 
	% $W$ is a closed half space in $V_1$.
	there exists a non-zero element $\lambda$ in the dual of $V_1$ such that
	\begin{equation}
	\label{halfspace}
	W= \lambda^{-1}( [ 0, +\infty)) \subseteq V_1.\end{equation}
	
	If $W$ is a 
	horizontal half-space defined by $\lambda\in (V_1)^*$ as in \eqref{halfspace}, then we define its (horizontal) boundary as
	%   and in this case we define
	$$\partial W\coloneqq \lambda^{-1}( \{ 0\}).$$ 
	% which we called the (horizontal) boundary of $W$
	Notice that $W$ is a   closed subset of $V_1$ and $\partial W$ is its boundary within $V_1$, which in our case will always contain $0 \in\g$. Observe that $\partial W$ is a hyperplane in $V_1$.
	
	Given a subset $W$ of $\g$, which we shall usually assume to be a horizontal half-space, the 
	semigroup $S_W$ generated by $\exp(W)$ is   described as 
	\begin{equation}
	\label{def:SW}
	S_W = \bigcup_{k=1}^\infty (\exp(W))^k,\end{equation}
	where
	\begin{equation*}
	\label{(exp(W))^k}
	(\exp(W))^k \coloneqq
	\{ \Pi_{i=1}^k \exp(w_i)\mid   w_1,\ldots, w_k\in W  \}.\end{equation*}
	Be aware that, even when $W$ is  closed (within $V_1$), the set  $S_W$ may not be closed within $\exp(\g)$.
	%, as shown by the following example. Consider, in exponential coordinates, the standard 3-dimensional Heisenberg group $ \mathbb H = \R^3 $ with the group product
%	\[
%	x\cdot y = (x_1+y_1,x_2+y_2,x_3+y_3+\half(x_1y_2-x_2y_1)).
%	\]
%	Then we claim that for the horizontal half-space $W \coloneqq \{x_1\geq 0, x_3=0 \} $ we have that $S_W = \{x_1>0 \}\cup \{0\} $. Indeed, if $ x\in S_W $, then by definition of $ S_W $ the point $ x $ has expression $ x= x^1\cdots x^k $ for some $ x^i\in W $. If $ x $ is such that $ x_3 \neq 0$, then necessarily $ x_1^i \neq 0 $ for some $ i= 1,\dots,k $. But since $ x_1^i\geq0 $ for every $ i= 1,\dots,k $, we deduce that $ x_1 > 0 $ and so $S_W \sus \{x_1>0 \}\cup \{0\} $. Regarding the other inclusion, it is straightforward to verify that for every $ x \in  \{x_1>0 \}\cup \{0\}  $ there exist some $ x^1,x^2\in W $ such that $ x = x^1\cdot x^2 $. Hence, $ S_W $ is not closed and $ S_W \cap [\mathbb{H},\mathbb H] = \{0\} $, but the half-space $ W $ is indeed ample since $ \Cl(S_W) = \{x_1 \geq 0 \} = W \oplus [\mathbb{H},\mathbb H]$.
	
	A vector space $\h$ of a stratified Lie algebra $\g$ is said to be {\em homogeneous} if there exist subspaces $\h_i$ of $V_i(\g)$ such that $\h=\h_1\oplus \ldots \oplus h_s$. Equivalently, we have that $\h$ is 
	homogeneous
	if and only if $\delta_\lambda \h = \h$ for all $\lambda>0$,  where $\delta_\lambda$ is the Lie algebra automorphism such that $\delta_\lambda (v)=\lambda v$ for $v\in V_1(\g)$.
	 We shall frequently use the fact that the center of a stratified Lie algebra is homogeneous:
	\begin{equation}\label{center:homog}
	\Center(\g) = \big(V_1(\g)\cap \Center(\g)  \big)\oplus \ldots \oplus\big(V_s(\g)\cap \Center(\g)  \big)\,.
	\end{equation}

\subsection{Lemmata in arbitrary algebras}
	In this subsection, let $ \g $ be a  Lie algebra of a simply connected Lie group $G$.
	We assume that $\g$ is  stratified  with nilpotency step equal to $s$. 
	Since $G$ is consequently nilpotent and simply connected, the exponential map $\exp: \g \to G$ is a bijection. We then have a correspondence between subsets 
	$\s\sus \g$ and subsets $S=\exp(\s)\subseteq G$.
	
	\begin{definition}\label{def:wedge}
		We associate with every subset $\mathfrak s \subseteq \g$ the following two sets
		\begin{equation}\label{wedge}
		\mathfrak w _\s\coloneqq \{X\in \mathfrak g \; : \; \R_+ X \subseteq \s \}.
		\end{equation}
		\begin{equation}\label{inv0}
		\edge{\s} \coloneqq \mathfrak w _\s \cap (-  \mathfrak w _\s) = \mathfrak w _\s \cap \mathfrak w _{-\s} .
		\end{equation}
		The set $ \mathfrak w _\s$ is known  as the {\em  tangent wedge of $\s$}
		and $ \edge{\s}$ as the {\em  edge of the wedge} $ \mathfrak w _\s$, see \cite[Page 2 and page 19]{Hilgert_Neeb:book_semigroup}. An equivalent definition for $ \edge{\s}$ is
		\begin{equation}\label{inv}
		\edge{\s}= \{X\in \mathfrak g \; : \; \R X \subseteq \s \}.
		\end{equation}
	\end{definition}
 Regarding the next result, we claim very little originality.
The arguments are mostly taken from \cite{Hilgert_Neeb:book_semigroup} and \cite{Bellettini-LeDonne2}.
% Should we here mention that (at least some of) the results are not original (exist already in \cite{Hilgert_Neeb:book_semigroup}), and that also half of the proofs are given in \cite{Bellettini-LeDonne2}?}
	Also,
%	In the next result, 
	the notions of 
	cone and
	convexity that we shall use are the usual ones with respect to the vector-space structure of the Lie algebra.
	\begin{lemma}\label{lem:CH}
		Let $G$ be a Lie group whose exponential map $\exp:\g\to G$ is injective.
		Let $\mathfrak s \subseteq \g$ be such that $\exp(\mathfrak s)$ is a  semigroup.
		Then  the     sets $ \edge{\s}$ and $ \mathfrak w _\s$, defined in \eqref{inv} and \eqref{wedge}, respectively, satisfy the following properties:
		\begin{enumerate}
			\item $ \mathfrak w _\s$ is the largest cone in $\s$;
			\item    $   \edge{\s}$ is the largest subalgebra of $\g$ contained  in $\s$;
			\item  for each $X\in \s\cap(-\s)$, we have that $\s$,  $ \mathfrak w _\s$, and $  \edge{\s}$ are invariant under $ e^{\ad_X} $, i.e.,
			\begin{equation}\label{ad_inv0}
			e^{\ad_X}     \s=     \s, \qquad \text{ for all } X \text{ such that } \pm X\in \s,
			\end{equation}
			\begin{equation}\label{ad_inv}
			e^{\ad_X}    \mathfrak w _\s=     \mathfrak w _\s, \qquad \text{ for all } X \text{ such that } \pm X\in \s,
			\end{equation}
			\begin{equation}\label{ad_inv2}
			e^{\ad_X}   \edge{\s}=    \edge{\s}, \qquad \text{ for all } X \text{ such that } \pm X\in \s;
			\end{equation}
			\item  if $\exp(\mathfrak s)$ is closed, then $ \mathfrak w _\s$ is closed and  convex.

		\end{enumerate}
		% and $\mathfrak c \subseteq \log(  S)$ is a cone, then
		% the closure of the convex hull of $\mathfrak c $ is a subset of $ \log(  S)$.
	\end{lemma}

	\begin{proof}
		%Set $\mathfrak s:=\log(  S)$. 
		Point (1) is immediate  from the definition.
		Regarding (2), to see that $\edge{\s}$ is a Lie algebra,  let $\mathfrak h$ be the Lie algebra generated by $   \edge{\s}$. Let $\hat S$ be the semigroup generated by $\exp(   \edge{\s})$. Since $   \edge{\s}$ Lie generates $\mathfrak h$, then by \cite[Theorem 8.1]{AS} the set $\hat S$ has nonempty interior  in $\exp(\mathfrak h)$. Since  $ \edge{\s}$ is symmetric, then   $\hat S$ is closed under inversion, hence a group. Being a group with nonempty interior, $\hat S$ is an open subgroup of $\exp(\mathfrak h)$.
		Since $\hat S$ is an open subgroup of the connected group  $\exp(\mathfrak h)$, then 
		$\exp(\mathfrak h)  $ equals $ \hat S$, which is a subset of $\exp(\s)$. Therefore, $\mathfrak h$ is a  subset of $\s$, being $\exp$ injective. Since in addition $ \mathfrak h $ is symmetric, we infer that $\mathfrak h\subseteq  \edge{\s}$, which tells us that $ \edge{\s}$ is a subalgebra of $\g$ contained  in $\s$. It is the largest since, if a Lie algebra $\mathfrak h$ is contained in $\s$, then from $\R \mathfrak h=\mathfrak h$ we deduce that $\mathfrak h \subseteq  \edge{\s}$.
		
		To prove \eqref{ad_inv0}, take $X $ such that $\pm X\in \mathfrak s$, so that for all $Y\in \s$ we have
		\begin{eqnarray*}\label{rem:conj}
			\exp(e^{\ad_X}Y)&=&
			\exp(\Ad_{\exp(X)}Y)\\& =& \exp((C_{\exp(X)})_*Y)  \\&=& C_{\exp(X)}(\exp(Y))\\&=&
			{\exp(X)}\exp(Y)\exp(-X)\in S,
		\end{eqnarray*}
		where we have used that 
		$
		\ad$ is   the differential of $
		\Ad$, that
		$
		\Ad_g$ is the differential of $C_g$, that exp intertwines  this differential with $C_g$ and, finally, that $S$ is a semigroup.
		Hence,  we have proved that 
		$ e^{\ad_X}\mathfrak s =\mathfrak s$.
		Consequently, since the map  $e^{\ad_X}$ is linear, it sends 
		half-lines to half-lines and lines to lines. Thus, we have \eqref{ad_inv} and \eqref{ad_inv2}.
		
		%cones in $\mathfrak s$ to cones in $\mathfrak s$. In particular, this map fixes the largest cone, i.e., it fixes $ \mathfrak w _S$, and because of \eqref{inv0}, it also fixes  $  \edge{\s}$.

		We now prove (4). If $ \exp(\mathfrak s) $ is closed, then also $\mathfrak s$ is closed since $ \exp $ is continuous and injective. Then the closure of $ \mathfrak w _\s$ is a cone in $\mathfrak s$. By maximality of 
		$ \mathfrak w _\s$, we deduce that 	$ \mathfrak w _\s$ is closed.	
		Since $\mathfrak w _\s  $ is a cone, to check that $ \mathfrak w _\s$ is convex
		%Since $ \log(  S)$ is closed, we just need to prove that the convex hull of $\mathfrak c $ is a subset of $ \log(  S)$. 
		 it is enough to show that $X+ Y\in \s  $ for all $X, Y\in \mathfrak w_\s $. Indeed, noticing that also $ \R_+X, \R_+ Y \sus \mathfrak w_\s $, this would imply that $ \R_+(X+Y)\sus \s $ 
	 	and so $ X+Y \in \mathfrak w_\s $. To prove that  $X+ Y\in \s  $ for every  $X, Y\in \mathfrak w_\s $, recall the formula, which holds in all Lie groups, % by \cite{}, 
		\begin{equation}
		\label{sum}
		\exp(X+Y) = \lim_{n\to \infty} \left(\exp\left(\tfrac1n X\right) \exp\left(\tfrac1n Y\right) \right)^n.
		\end{equation}
		Set $S=\exp(\s)$.
		Since  $\R_+ X, \R_+ Y \subseteq \s$, then 
		$\exp(\frac1n X), \exp(\frac1n Y) \in   S$, for all $n\in \N$.
		Consequently, since $  S$ is a semigroup, we have $ \left(\exp(\frac1n X) \exp(\frac1n Y) \right)^n \in   S$.
		Being $  S$ closed by assumption, we get from \eqref{sum} that $\exp(X+Y)  \in   S$. Since $\exp$ is injective, we infer that $X+ Y\in \s $. So the convexity of  $ \mathfrak w _\s$ is proved.
%		
%		
%		Regarding the last point, obviously, if a Lie algebra $\mathfrak h$ is contained in $\s$ then since $\R \mathfrak h=\mathfrak h$, then $\mathfrak h \subseteq  \edge{\s}$.
%		To see that $\edge{\s}$ is a Lie algebra, 
%		as a consequence of the convexity of  $ \mathfrak w _\s$, we have that $  \edge{\s}$ is closed under sum and scalar multiplication: its a vector subspace.
%		Hence, we only still need to check that it is closed under brackets. For doing this, take $X,Y\in \edge{\s}$.
%		Then by \eqref{ad_inv2}, we have that $e^{\ad_{(tX)}}Y \in \edge{\s}$ for all $t\in \R$. Being $\edge{\s}$ a vector space we have that
%		$\frac{1}{t} (e^{\ad_{(tX)}}Y- Y)  \in  \edge{\s}$.
%		Since we have $\frac{1}{t} (e^{\ad_{(tX)}}Y- Y) = [X,Y] + o(t)$, as $t\to0$, we get that $ [X,Y]   \in  \edge{\s}$.
	\end{proof}
We prove next a useful lemma, which states that if $ \R_+ X\sus \s,$ $ \R Y \sus \s, $ and  $ \R \ad_{Y}^2X \sus \s$ then also $ \R [X,Y]\sus \s$.  Recall the notions of  $ \edge{\s}$ and $ \mathfrak w _\s$, defined in \eqref{inv} and \eqref{wedge}.
	\begin{lemma}%[The first main lemma]
		\label{lemma:new_invariant_directions} Let $ \g $ be a stratified Lie algebra.
		Let $ \s\sus\g $ be a subset $\exp(\mathfrak s)$ is a closed semigroup.  
  If 	$X \in \mathfrak w _\s $ and $Y\in \edge{\s}$ are such that $\ad_{Y}^2X \in \edge{\s}$, then $ [X,Y]\in \edge\s$.  
 \end{lemma}
	
	\begin{proof} 
%	Part 1 is an immediate consequence of the fact that, for the closed set $\bar\s:= \Cl_\g (\s) $ we have that  $\exp(\bar\s)$ is a closed semigroup and hence
%		$\inv{\bar\s}$ is a subalgebra. Thus if
%		$X\in  \edge{\s}\subseteq \inv{\bar\s},   Y \in  \edge{\s}\subseteq \inv{\bar\s} $, 
%		then   $   [X,Y]\in   \inv{\bar\s}$ and $   X+Y\in   \bar\s$.
		One the one hand, since $ \edge{\s}$ is a Lie algebra by Lemma~\ref{lem:CH}.(2) we have  $\ad_{Y}^kX \in  \edge{\s}$, for all $k\geq 2$.
		On the other hand, from \eqref{ad_inv} we have that 
		$e^{\ad_{tY}}X\in  \mathfrak w _{ \s}$, for all $t\in \R$. Hence, since $ \mathfrak w_\s$ is convex by Lemma~\ref{lem:CH}.(4), for all $t\in \R$ we have 
		$$  	X   + t [Y,X]    = e^{\ad_{tY}}X -  \sum_{k\geq 2}\dfrac{t^k}{k!}\ad_{Y}^{k} (X)  \in \mathfrak w _{ \s}.
		$$
		Hence $\frac{1}{|t|}(X   +  t [Y,X]  )   \in \mathfrak w _{ \s} $, for all $t\in \R$. Therefore, taking $t$ to $\pm\infty$ and using that $ \mathfrak w_\s $ is closed and convex by Lemma~\ref{lem:CH}.(4), we get  $ [Y,X]     \in  \edge{\s} $. 
	\end{proof}
	
	In the rest of the paper, we focus on semigroups generated by horizontal half-spaces in stratified Lie algebras.
	For every horizontal half-space $W$, see \eqref{halfspace}, in a stratified Lie algebra $  \g$, 
	we denote by $S_W $ the semigroup generated by $\exp(W)$ in $\exp(\g)$, see \eqref{def:SW}, and by 
	$\mathfrak s_W\subset \g$ the set such that
	$\exp(\mathfrak s_W)=S_W,$ i.e.,
	\begin{equation}\label{s_W}
	\mathfrak s_W\coloneqq\log(S_W).
	\end{equation}

If $\s \coloneqq \s_W$, we stress the following two immediate facts:
\begin{equation}
	  \text{ for every } X \in V_1 ,   \text{ either } X \in \mathfrak w_\s  \text{  or } -X \in \mathfrak w _\s ;
	  \end{equation}
	  	  \begin{equation}
  \partial W = \edge{\s}\cap V_1 .
	  \end{equation}
The semigeneration condition stated in the introduction can equivalently be defined as follows: A   set $W$ in $\g$ is  {\em   semigenerating} if 
	\begin{equation}\label{ample:def}
	[\g,\g]\subseteq \Cl(\mathfrak s_W ),
	\end{equation}
 and we say that   $\g$ is  {\em   semigenerated} if every horizontal half-space $W$ in $\g$ is      semigenerating.

\begin{remark}For  a  horizontal half-space $ W \subseteq \g$ and for $\s $ equal to $ \s_W $ or $ \Cl(\s_W)$,  we have that
$   \edge{\s}$ is a homogeneous subalgebra of $\g$ contained  in $\s$.
Indeed, in Lemma~\ref{lem:CH}.(2) we already proved everything except the homogeneity. In such a case, for all $\lambda>0$ we have that 
$\delta_\lambda W=W$ and, hence, $\delta_\lambda \s=\s$. Thus we infer that 
$\delta_\lambda \R X\subseteq \s$ if and only if
$   \R X\subseteq \s$. Therefore $   \edge{\s}$ is homogeneous.  
\end{remark}

	\begin{lemma}\label{lemma:ample_directions}
		Let $ \g $ be a stratified Lie algebra and $ W \subseteq \g$ a  horizontal half-space. 
		Then the set $\s = \Cl(\s_W)$  has the following two properties:
%		Denoting by $ \s $ the closed ... associated to $ W $, we have the following.
%		\begin{enumerate}
%			\item
%			If $ X,Y \in V_1 $ are such that $ \ad_X^2 Y = \ad_Y^2 X = 0 $, then 
			\begin{equation} \label{lemma:universal_invariant}
			X,Y \in V_1 \text{ with } \ad_X^2 Y = \ad_Y^2 X = 0 \text{ implies }
			 [X,Y] \in \edge \s;
			\end{equation}
%			\item  We have 
			\begin{equation}
			\label{l:V2:cap:Z(g):univ:invariant} V_2 \cap \Center(\g) \sus \edge \s.
			\end{equation}
%		\end{enumerate}
	\end{lemma}
	
	\begin{proof} Regarding \eqref{lemma:universal_invariant},
%		fix any horizontal half space $W\subseteq V_1$. W
		we have, up to changing signs, that
		$X,Y\in W$. Moreover, 
		since $\partial W$ is a codimension 1 subspace of $ V_1 $, we have that,
		up to possibly swapping $X$ with $Y$, there exists some $ a\in \R $ for which 
		$Z\coloneqq  Y - aX  \in \partial W$. 
		Thus, we have
		$Z \in \partial W\subseteq \edge{\s}$  and $X\in W\subseteq\mathfrak w_\s$.
		Moreover, by the assumptions on $X$ and $Y$, we have that
		\[
		\ad_{Z}^2 X = [Y-aX,[Y,X]] = \ad_Y^2 X + a\, \ad_X^2 Y = 0  \in   \edge{\s}.
		\]
		By Lemma~\ref{lemma:new_invariant_directions}, we obtain $ \edge{\s}\ni [X,Z] = [X,Y] $.

		Regarding \eqref{l:V2:cap:Z(g):univ:invariant},
		we choose $ \{X_1,\ldots,X_m\} $ to be a basis of $ V_1 $ such that $ X_1 \in W\subseteq \mathfrak  w _\s $ and $ X_2,\dots,X_m \in \partial W\subseteq\edge{\s}$. Take $ Z \in V_2\cap \Center(\g) $ and express it, for some $ a_i, b_{ij}\in \R $, as
	\begin{equation}\label{decompongo_Z}
	Z = \sum_{i\geq 2}a_i[X_i,X_1] + \sum_{i,j\geq 2}b_{ij}[X_i,X_j] = [Y,X_1]+\tilde Y,
		\end{equation}
	where 
	\[ 
	Y \coloneqq  \sum_{i\geq 2}a_iX_i  \quad \text{and}\quad \tilde Y \coloneqq \sum_{i,j\geq 2}b_{ij}[X_i,X_j]  .
	\]
	Since $ \edge \s $ is a Lie algebra by Lemma~\ref{lem:CH}.2, we have that the elements $ Y, \tilde Y, [Y,\tilde Y] $ belong to $ \edge \s $. Since $ Z \in \Center (\g) $, we get also
	\[
	0 = [Y,Z] = \ad_{Y}^2X_1 + [Y,\tilde Y],
	\]
	which implies that $ \ad_{Y}^2X_1 \in \edge \s $. 
	Since $ X_1 \in  \mathfrak  w _\s $, $ Y\in \edge \s $, and  $ \ad_{Y}^2X_1 \in \edge \s $
	 Lemma~\ref{lemma:new_invariant_directions} tells us that $ [Y,X_1]\in \edge \s $.
	 Going back to \eqref{decompongo_Z}, we finally infer that $ Z \in \edge \s $, again because $ \edge \s $ is a Lie algebra by Lemma~\ref{lem:CH}.2.
	\end{proof}
	
 	For the next lemma, recall that $\pi_{\mathfrak i}:\g \to \g/\mathfrak i $ is the quotient map modulo an ideal $\mathfrak i$. We also recall the basic fact that in the Lie algebra $\g$ of a simply connected nilpotent Lie group $G$, a subset $\mathfrak i \subseteq \g$ is an ideal if and only if $N\coloneqq\exp(\mathfrak i)$ is a normal Lie subgroup of $G$; in this case, the quotient $ \g/\mathfrak i $ is canonically isomorphic to the Lie algebra of $ G/N $ and we have the following commutative diagram:
		\begin{center}
			\begin{tikzcd} %[end anchor=north]
			\g \arrow[d, "\exp" ']\arrow[r, "\pi_{\mathfrak i}"] & \g/\mathfrak i \arrow[d, "\exp"]\\
			G\arrow[r, "\pi_{N}"]& G/N .
		\end{tikzcd}
			
	\end{center}
Moreover, if $ \g $ is stratified, then $ \g/\mathfrak i $ canonically admits a stratification if and only if $ \mathfrak i $ is homogeneous.

We stress   that we have the following fact for each subset $W\subseteq \g$ of a  Lie algebra $ \g $:
\begin{equation}\label{quotient_semigr}
\pi_{ \mathfrak i}(\s_W)   =  \s_{\pi_{ \mathfrak i} (W )}.
\end{equation}
Indeed, setting $N\coloneqq\exp( \mathfrak i)$ and denoting by $S(A)$ the semigroup generated by $A$,  we need to show that $\pi_{ N}(S(\exp(W)))   =  S({\pi_{N}\exp (W )}).$
In fact, on the one hand, since the homomorphic image of a semigroup is a semigroup, we have that
$\pi_{ N}(S(\exp(W)))$
%$\exp(\pi_{ \mathfrak i}(\s_W) ) = \pi_{N}(S_W)  $ 
is a semigroup containing $ \pi_{N}(\exp(W) )$, so $ S({\pi_{N}\exp (W )}) \sus \pi_{ N}(S(\exp(W)))$.
%=\exp(\pi_{ \mathfrak i} (W )) $. Hence
%$
%\pi_{ \mathfrak i}(\s_W)   \supseteq   \s_{\pi_{ \mathfrak i} (W )}$.
 On the other hand, the set 
 $\pi_{ N}(S(\exp(W)))=   S(\exp(W)) N $ %
% $\exp(\pi_{ \mathfrak i}(\s_W) ) = \pi_{N}(S_W) =\exp(  \s_W)  N =\exp(  \s_W)  \exp \mathfrak i   $
 is contained in the semigroup generated by 
 $\exp(   W)N = \pi_{N} (\exp( W )  ) $, {i.e.}, we have $ \pi_{ N}(S(\exp(W)))\sus  S({\pi_{N}\exp (W )})$.
% $\exp(   W+ \mathfrak i)=\exp(  \pi_{ \mathfrak i} (W )  ) $.
%Hence
%$
%\pi_{ \mathfrak i}(\s_W)   \subseteq   \s_{\pi_{ \mathfrak i} (W )}$.

\begin{lemma}\label{lemma:quotient_of_cones_by_invariants}
	Let $ \mathfrak i $ be a homogeneous ideal of a stratified Lie algebra $ \g $ and let $ W \subseteq\g $.
	\begin{itemize} 
	\item[(i)]
	If   $ W $ is semigenerating, then $ \pi_{ \mathfrak i} (W ) $ is semigenerating.
	\item[(ii)]
	If $ \mathfrak i\sus \Cl(\s_W) $ and $ \pi_{ \mathfrak i} (W ) $ is semigenerating, then $ W $ is semigenerating.
		\end{itemize}
\end{lemma}
\begin{proof}
%Let $N\coloneqq\exp \mathfrak i$. 

	Assume first that $ W $ is semigenerating.
	Then from \eqref{quotient_semigr} we obtain that $\pi_{ \mathfrak i} (W )$ is semigenerating by the following calculation:
	% If $ \mathfrak i \cap V_1 \nsubseteq W$, then $V_1 \sus W + \mathfrak i $ and $ S_{W+\mathfrak i} = S_{V_1} = G $ in which case $ W + \mathfrak i $ is trivially ample. If $ \mathfrak i \cap V_1 \sus W $, then
	\[
	[ \pi_{ \mathfrak i} (\g) ,\pi_{ \mathfrak i}(\g) ] =\pi_{ \mathfrak i} ( [\g,\g]) 
%	= [\g,\g]+\mathfrak i
	   \stackrel{\eqref{ample:def}}{\sus}
%	  \Cl(\s_W) +\mathfrak i =
	\pi_{ \mathfrak i}(\Cl(  \s_{ W }))\sus
	   \Cl(\pi_{ \mathfrak i}(\s_W)  ) \stackrel{\eqref{quotient_semigr}}{=}\Cl(  \s_{\pi_{ \mathfrak i} (W )})
 .
	\]
	
	Suppose then that  $ \pi_{ \mathfrak i} (W ) $ is semigenerating and that $ \mathfrak i \sus \Cl(\s_W) $.
	Then we also have the containment
	$
 N\coloneqq\exp(\mathfrak  i)\sus \Cl(S_W) 
 $.  Since $ \Cl(S_W) $ is a semigroup, we have 
		\begin{equation}\label{intermezzo}
		 \Cl(S_W)\cdot N=\Cl(S_W) . 	\end{equation}
		  Therefore  from \eqref{quotient_semigr}  we get  
	\begin{equation}\label{intermezzo2}
	S_{\pi_{\mathfrak i}(W) } = 	
	\exp(\s_{\pi_{\mathfrak i} (W)})
	\stackrel{\eqref{quotient_semigr}}{=}	
	\exp(\pi_{ \mathfrak i}(\s_W) )=	
	\pi_{N}(\exp(\s_W)) =	
	\pi_{N}(S_W)  .
		\end{equation}
	Taking the closure and the preimage under $\pi_{N}$,  from the fact that $\pi_N$ is an open map (and hence $\pi_{N}^{-1}$ and $\Cl$ commute) and from  \eqref{intermezzo}, we get that
	\begin{equation}\label{eq:(A)}
	\pi_{N}^{-1}\Cl(S_{\pi_{\mathfrak i}(W) }) 
	\stackrel{\eqref{intermezzo2}}{=} \pi_{N}^{-1}\Cl(\pi_{N}(S_W))
	=	
	\Cl(S_{W}\cdot N) = \Cl(S_W)\cdot N \stackrel{\eqref{intermezzo}}{=}  \Cl(S_W).
	\end{equation}
%	where we used in the first equation that $ \exp(X + \mathfrak i)=\exp(X)\exp(\mathfrak i) $ for and ideal $ \mathfrak i $ of $ \g $, and so the semigroup generated by $ W + \mathfrak i $ equals to $ S_W \cdot \exp(\mathfrak i)$.
	 Consequently, taking the logarithm,
	 \begin{equation}\label{eq:(B)}
	 	\pi_{\mathfrak i}^{-1}  \Cl(\s_{\pi_{\mathfrak i}(W)}) =
	 \log (  \pi_{N}^{-1}\Cl(S_{\pi_{\mathfrak i}(W) })   )
	 \stackrel{\eqref{eq:(A)}}{=} \log  \Cl(S_W) = \Cl(\s_W).
	 \end{equation}
	%Notice that since $ \mathfrak  i \sus \Cl(\s_W) $, then in particular $ \mathfrak i \cap V_1 \sus \Cl(\s_W)\cap V_1 = W $. 
	Hence, since $ \pi_{ \mathfrak i} (W ) $ is semigenerating, we infer  $$ 	[\g,\g] \sus [\g,\g]+\mathfrak i 
	= \pi_{\mathfrak i}^{-1}  [ \g/ \mathfrak i , \g/ \mathfrak i ] \sus
\pi_{\mathfrak i}^{-1}  \Cl(\s_{\pi_{\mathfrak i}(W)})
  \stackrel{\eqref{eq:(B)}}{=} \Cl(\s_W), $$ proving that $ W $ is semigenerating.
%	\[
%	\mathfrak i = ( \mathfrak i\cap V_1 )\oplus (\mathfrak i\cap[\g,\g]) \sus W\oplus[\g,\g]
%	\]
%	and so 
\end{proof}

 	We keep reminding that a quotient algebra $ \g/\mathfrak i $ of a Carnot algebra $ \g $ is Carnot if and only if the ideal $ \mathfrak i $ is homogeneous. In such a case, we say that $ \g/\mathfrak i $  is a {\em Carnot quotient} of $\g$.
	
	\begin{proposition} 
	\label{prop:quotients_product_ample}
	Carnot quotients and products of semigenerated algebras are semigenerated. 
	\end{proposition}
	\begin{proof} Consider a quotient algebra $\g/ \mathfrak i $ of a semigenerated Carnot algebra $\g$ by a homogeneous ideal $ \mathfrak i $. Then, by Lemma~\ref{lemma:quotient_of_cones_by_invariants}.i, the Carnot algebra  $\g/ \mathfrak i $ is semigenerated, since every horizontal half-space in $\g/ \mathfrak i $ is of the form $ \pi_{ \mathfrak i} (W ) $ for some horizontal half-space  $ W \subseteq\g $.
	
	Regarding products, let
	 $\g$ be a Carnot algebra that is the direct product 
	 $\g=\g_1\times \g_2  $ of two of its Carnot subalgebras.
	   Assume that
	 $\g_1$ and $\g_2$ are semigenerated. Let 
	$ W \subset \g   $ be a horizontal half-space.
Then for each $i=1,2$ we have that the set $ W_i\coloneqq W \cap V_1(\g_i)$ is a horizontal half-space in $ V_1(\g_i)$, or possibly the whole of $ V_1(\g_i)$. Since each $ \g_i $ is semigenerated, 
$[\g_i,\g_i] \subseteq \bar  \s_{W_i}$ and hence $[\g_i,\g_i] \subseteq \mathfrak{e}({\bar  \s_{W_i}})$ as $  \mathfrak{e}({\bar  \s_{W_i}}) $ is the largest subalgebra of $ \bar  \s_{W_i} $ by Lemma \ref{lem:CH}.2. Consequently,
$$ [\g_1,\g_1]\cup [\g_2 ,\g_2]   \subseteq   \mathfrak{e}({\bar \s_{W_1} })  \cup  \mathfrak{e}({\bar  \s_{W_2} })   \subseteq  \mathfrak{e}({ \bar  \s_{W}  }) .$$
As $  \mathfrak{e}({\bar  \s_{W} })$ is a vector space,  we have that also
\[
[\g ,\g ] = [\g_1,\g_1]\times [\g_2 ,\g_2] = \spann{ [\g_1,\g_1]\cup [\g_2 ,\g_2] } \sus \bar  \s_{W}.
\]
Hence we infer that $W$ is semigenerating.
	\end{proof}
 
\begin{remark}\label{rmk:filiforms_engel_quots}
	As a direct consequence of Proposition~\ref{prop:quotients_product_ample}, one observes that if a Carnot algebra has a non-semigenerated Carnot quotient, then the algebra cannot be semigenerated. In particular, we point out that every rank-2 Carnot algebra of step at least 3 is not semigenerated, since it has the Engel algebra as a quotient. Indeed, for every such algebra $ \g $,  
%	with two-dimensional first layer and of step $ \geq $ 3, 
	we have that $ \g/\g^{(4)} $ with $ \g^{(4)}  = V_4 \oplus \dots \oplus V_s$ is a rank-2 Lie algebra of step exactly 3, i.e., either the Engel algebra or the free Lie algebra of rank 2 and step 3. Since the Engel algebra is a quotient of the free Lie algebra, the claim follows.
	Regarding the fact that the Engel algebra is not semigenerated, we refer to Section~\ref{def:nthEngel} and specifically to Proposition~\ref{prop:Engels:non-CNP}.
\end{remark}

	In the next proposition we verify that three conditions for a stratified Lie algebra are equivalent. In the rest of the paper, we shall call {\em trimmed} every such Lie algebra. 
	%\begin{defn}
	%	A stratified Lie algebra is called \emph{minimal} if its \end{defn}
	\begin{proposition}[Equivalent conditions for the definition of trimmed algebra]\label{lemma:minimal_TFAE}
		For a stratified Lie algebra  $ \g $  the  following are equivalent:
		\begin{itemize} 
			\item[(a)]  every proper quotient of $ \g $ has lower step;
			\item[(b)] $ V_s \sus \mathfrak i $ for every nontrivial ideal $ \mathfrak i $ of $  \g $, where $ s $ is the step of $g$;
			\item[(c)] $ \dim \Center(\g) = 1 $.
		\end{itemize}
		%	When these equivalent conditions are satisfied we say that $\g$ is {\em minimal}.
	\end{proposition}
\begin{proof}
The fact that (a) and (b) are equivalent comes from the correspondence between ideals and kernels of homomorphisms. If every ideal contains the last layer, then any quotient has lower step. Vice versa, if there exists an ideal that does not contain  the last layer, then the  quotient modulo that ideal has still step $s$.  

	To see that (b) implies (c),   suppose by contradiction that $ \dim \Center(\g) > 1 $. 
	We consider the two cases: $ \dim V_s > 1$ or $ \dim V_s =1$. In the first case, we get a contradiction since
	every one-dimensional subspace $\mathfrak i$ of $V_s$ is a  nontrivial ideal  of $  \g $ for which $ V_s \sus \mathfrak i $ is not true.
	In the case $ \dim V_s =1$, recalling that $ \Center(\g)$ is graded by \eqref{center:homog}, we get that
%	Then there is a nontrivial $ X \in Z(\g) $ such that $ \g / \langle X \rangle $ has the same step as $ \g $. Indeed, if $ Z(\g) \sus V_s $, then $ \dim V_s > 1$ and there is a nontrivial proper subspace of $ V_s $. 
	 $ \Center(\g) \cap (V_1 \oplus \dots \oplus V_{s-1})   $
	 is a  nontrivial ideal  of $  \g $ for which $ V_s \sus \mathfrak i $ is not true. These contradictions prove that (b) implies (c).

%	  we may choose $ X $ in the intersection.
	
To see that (c) implies (b),  let $ \mathfrak i \sus \g $  be a nontrivial ideal.
Since $\g$ is nilpotent, we have\footnote{Looking at the sequence $\ad_\g^j(\mathfrak i )$ one finds a non trivial subset of $\mathfrak i $ that commutes with $\g$.}
 that    $ \mathfrak i \cap \Center(\g) \neq \{ 0 \}$. Since $ \dim \Center(\g) = 1 $, we have $ \Center(\g) \sus\mathfrak i  $. Finally, since $ V_s \sus \Center(\g) \sus \mathfrak i $  we get the claim.
\end{proof}
	\begin{definition}\label{minimal}
		If $\g$ is a stratified Lie algebra that satisfies the
		equivalent conditions of Proposition~\ref{lemma:minimal_TFAE},
		then we say that $\g$ is {\em trimmed}.
		
	\end{definition}
	
We expect that every non-semigenerated algebra has a trimmed non-semigenerated quotient. However, we only prove the following weaker statement, which will suffice in the step-3 case.
	
\begin{prop}\label{prop:ample_quotient_small_center}
	Let $ \g $ be a	stratified Lie algebra.	If $\g$ is not semigenerated, then there exists a quotient algebra $\hat \g$ of $ \g $ that is  not semigenerated such that $ \Center(\hat\g)\cap V_j(\hat \g) = \{0\} $ for $ j =1,2 $, and $ \dim \Center(\hat \g) \cap V_j(\hat \g)\leq 1$ for all $j=3,\ldots , s$.
\end{prop}
\begin{proof}
	Let $ W\sus \g $ be a non-semigenerating horizontal half-space. Replacing $ \g $ with some quotient of it, we may suppose that for every proper homogeneous ideal $ \mathfrak  i $ of $ \g $ the half-space $ W/\mathfrak i $ is semigenerating in $ \g/\mathfrak i $. Indeed, if there exists a homogeneous ideal $ \mathfrak i $ of $ \g $ such that $ \g/\mathfrak i $ is not semigenerated, we replace $ \g $ with $ \g/\mathfrak i $. We repeat this procedure until every homogeneous ideal $\mathfrak  i $ of $ \g $ has the property that $ \g/\mathfrak i $ is semigenerated. This will happen, eventually, since every step-2 Carnot algebra is semigenerated.
	
	We shall then show that $\g$ has the required properties.
	 First, we check that $ \Center(\g)\cap V_1 $ is trivial. Indeed, if this space is nontrivial, then, for some $ n\geq 1 $,
	 $$ \g \cong (\g /(\Center(\g)\cap V_1))\times (\Center(\g)\cap V_1)\cong (\g /(\Center(\g)\cap V_1))\times\R^n. $$
	 Since the product of two semigenerated Lie algebras is semigenerated (see Proposition~\ref{prop:quotients_product_ample}), we get a contradiction.
	 
	 To prove that $ \Center(\g)\cap V_2 $ is trivial, recall that $  \Center(\g)\cap V_2 \sus \pedge{\Cl(\s_W)} $ by \eqref{l:V2:cap:Z(g):univ:invariant}. Then, denoting by $ \pi $ the projection $ \pi \colon \g \to \g/( \Center(\g)\cap V_2 ) $, by Lemma~\ref{lemma:quotient_of_cones_by_invariants}.ii we have that $ \pi(W) $ is not semigenerating. Since we assumed that $ W/\mathfrak i $ is semigenerating for every proper ideal of $ \g $, we deduce that $ \Center(\g)\cap V_2 = \{0\} $.
	 
	 Fix any $ j\geq 3 $. Assume by contradiction that $ \dim \Center(\g)\cap V_j > 1$. Take two linearly independent homogeneous vectors $ v_1,v_2\in \Center(\g)\cap V_j $. Let us denote by $ V $ the 2-dimensional homogeneous subspace $ \spann{v_1,v_2} $. Let us also fix a scalar product on $ \g $ that makes $ v_1 $ and $ v_2 $ orthonormal and set $ \tilde \s \coloneqq \Cl(\s_W)\cap V$.
	
	Observe that, as $ V $ is central, each line $ \R v \in V $ is an ideal of $ \g $. Therefore, being $ W+\R v $  semigenerating in $ \g/\R v $, we have for all $ u,v\in V $  that, if $ \pi \colon \g \to \g/\R v $ stands for the projection,
	\[
	u + \R v\subset V \sus [\g,\g] \sus \pi^{-1}([\g/\R v,\g/\R v]) \sus \pi^{-1}(\Cl(\s_{\pi(W)})) = \Cl( \pi^{-1}(\s_{\pi(W)})) \stackrel{\eqref{quotient_semigr}}{=}\Cl(\s_W + \R v),
	\]
%	\[
%	u + \R v\subset V \sus [\g,\g] \sus \pi_{\R v}^{-1}([\g/\R v,\g/\R v]) \sus \pi_{\R v}^{-1}(\Cl(\s_{\pi_{\R v}(W)}) = \Cl \pi_{\R v}^{-1}(\s_{\pi_{\R v}(W)}) \stackrel{\eqref{quotient_semigr}}{=}\Cl(\s_W + \R v)
%	\]
	
	%we have that $ \Cl(\s_W + \R v) = \Cl(\s_{W + \R v}) $ is a half-space in $ \g / \R v $ and so $ [\g/\R v,\g/\R v]\sus \Cl(\s_W + \R v) $.  In particular, $ u + \R v \sus  \Cl(\s_W + \R v) $ for all $ u \in V $.
	\noindent
	where we again used that $ \pi^{-1} $ and the closure commute.
	 Hence, denoting by $  B_{1/n}(u + \R v) $ the $ 1/n $-neighborhood of the line $ u + \R v $ within $ \g $, we obtain
	\begin{equation}\label{obs:nbhd}
	\text{for every }  n\in \N  \text{ there exists } s_n \in B_{1/n}(u + \R v) \cap \s_W. %\text{where } B_{1/n}(u + \R v) \text{is the } 1/n \text{-neighborhood of the line } u + \R v .
	\end{equation}
	%for every $ n\in \N $ there exists $ s_n \in B_{1/n}(u + \R v) \cap \s_W$, where we denote by $  B_{1/n}(u + \R v) $ the $ 1/n $-neighborhood of the line $ u + \R v $.
	
%	There are two possibilities: either $ v_1 $ and $ v_2 $ have the same degree of homogeneity or not. We are going to treat these two cases separately. Assume first that %the degrees of  $ v_1$ and $v_2 $ are the same.
%	$ v_1$ and $v_2$ are homogeneous of the same degree. 
	
	We claim that   
	\begin{equation}\label{obs:every_halfspace}
	 \tilde \s \cap \S^1(V) \cap H \neq \emptyset \text{ for every closed half-space }H \sus V,
	\end{equation}% $ \s\cap\S^1(V) \cap H \neq \emptyset $ for every closed half-space $ H \sus V$. 
	where $ \S^1(V) $ stands for the unit circle of $ V $ with respect the restricted norm on $ V $.
	Indeed, denote by $ \partial H $ the boundary of $ H $, which is a line in $ V $, and let $ \nu $ be the inner unit normal of $ H $. Consider the sequence $ (s_n)_n $ given by \eqref{obs:nbhd} for the line $ 2\nu + \partial H \sus H$. Observe that, since $ V $ is spanned by homogeneous elements of the same degree, then for every $ u \in \g $ the projection of the continuous path $ \{\delta_\lambda(u) \mid 0<\lambda<1 \} $ on $ V $ is a straight line segment. Hence the projection of $ \{\delta_\lambda(s_n) \mid 0<\lambda<1 \} $ onto $ V $ is contained in $ H $ for every $ n\in \N $.
	% Moreover, $ \{\delta_\lambda(s_n) \mid 0<\lambda<1 \} $ is a continous path connecting $ s_n $ to the origin, and so there exists $ \lambda_n \in (0,1) $ for which $ \|\delta_{\lambda_n}(s_n) \| = 1 $. 
	Notice that, since $d(s_n,  2 \nu +\partial H)\leq 1/n$, then $ \| s_n  \| > 1 $, where we consider the norm induced by the fixed scalar product.
	Recall that $ \delta_\lambda $ is a continuous contraction for every $ \lambda \in (0,1) $. For that reason, for every $ n\in \N $, there exists $ \lambda_n \in (0,1) $ with $ \|\delta_{\lambda_n}(s_n) \| = 1 $. Moreover, still using the fact that $ \delta_{\lambda_n} $ is a contraction, from $ \mathrm{dist}(s_n,H)< 1/n $ we deduce that $ \mathrm{dist}(\delta_{\lambda_n}(s_n),\S^1(V)\cap H) < 1/n $ for each $ n $. 
	Being $ \S^1(V)\cap H $  compact and $ \s_W $ invariant under dilations, we find a converging subsequence of $ (\delta_{\lambda_n}(s_{n}))_n \sus \s_W$ with the limit in $ \S^1(V) \cap H$. Hence $ \tilde \s \cap \S^1(V)\cap H \neq \emptyset $, proving the claim \eqref{obs:every_halfspace}.
	
	Notice that since $ V $ is spanned by elements of the same degree of homogeneity and since $\tilde \s $ is dilation invariant, we have that   $\tilde\s $ is  a Euclidean convex cone.  We deduce that either $ \tilde\s  $ is contained in some closed half-space $ H \sus V $ or $ \tilde\s = V $. In the latter case $ V $ is an ideal of $ \g $ such that $ V \sus \Cl(\s) $, which implies by Lemma~\ref{lemma:quotient_of_cones_by_invariants}.ii that $ W/V $ is not semigenerating, contradicting our assumptions.
	
	We may then assume that there exists some closed half-space $ H \sus V $ such that $\tilde\s \sus H $. Let $ v $ denote one of the two intersection points of $\partial  H $ and $ \S^1(V) $. We are going to argue that $ v\in \tilde{\s} $. Consider a sequence $ (H_n)_n $ of closed half-spaces in $ V $ for which $ H_n\cap H \to \R_+ v $. By \eqref{obs:every_halfspace} we find a sequence $ (s_n)_n \sus \tilde{\s} \cap \S^1(V)$ such that each $ s_n \in H_n $. But since $\tilde\s \sus H $, we have that $ s_n \in H \cap H_n $ for every $ n $. Hence $ s_n \to v $ and $ v \in \tilde{\s} $. With a similar argument also $ - v\in \tilde{\s} $ and therefore $\{ \delta_\lambda(v)\mid \lambda\in \R \} = \R v \sus\tilde\s $.
	% by Euclidean convexity of $\tilde\s $. 
	Now $ \R v $ is again an ideal of $ \g $ contained in $ \Cl(\s) $, leading to a contradiction by Lemma~\ref{lemma:quotient_of_cones_by_invariants}.ii and the fact that $ W /\R v $ is semigenerating. 
\end{proof}

%\subsection{Other lemmas}
%
 The following lemma is an algebraic observation. It will be  essential in our proof of Theorem~\ref{t:suff:alg:cond:for:being:a:halfspace}, which is a refinement of   Theorem~\ref{thm:diamond}. 
\begin{lemma}%[The second main lemma]
	\label{subalg}
	Let $ \g $ be a stratified Lie algebra, 
%	of step $ s $ 
let $ W \subseteq  \g $ be a horizontal half-space,   and let $ \h $ be a subalgebra of $ \g $ containing $\partial W $. Then, the following conditions are equivalent.
\begin{enumerate}
	\item There exists $ X\in V_1 \setminus \partial W $ such that $ \mathrm{ad}^k_X Y \in \h$ for all $ Y \in \partial W $  and $ k\geq 1 $;
\item 	$ [\g,\g] \subseteq \mathfrak{h} $.
\end{enumerate}

	%$ \mathfrak{h} = (\partial W) \oplus V_2 \oplus \cdots \oplus V_s $.
\end{lemma}
\begin{proof}
%	\footnote{
The fact that (2) implies (1) is trivial, since if $ [\g,\g] \subseteq \mathfrak{h} $, then any choice of basis will satisfy the requirements.
For the opposite direction, without loss of generality we 
	  may assume that $ \g $ is a free nilpotent Lie algebra of step $ s $.
%	  , since any Lie algebra of step $ s $ is isomorphic to a quotient of the free Lie algebra. 
	  We shall consider a basis for $ \g $ that is constructed by a well-known  algorithm due to M. Hall \cite{Hall_basis}. Below we say that a vector $ Z $ has degree $ k $ if $ Z\in V_k $. 
	
%	Set $H\coloneqq\partial W$.
	 Let $ \{Y_1,\ldots, Y_m\} $ be a basis for $ \partial W $ and $ X\in V_1 \setminus \partial W $, whence $ \{X,Y_1,\ldots, Y_m\} $ is a basis for $ V_1 $. To construct the Hall basis, first fix an ordering for $ \{X,Y_1,\ldots, Y_m\} $ so that $ Y \geq X $ for all $ Y\in \{X,Y_1,\ldots, Y_m\}. $ Suppose then that we have defined Hall basis elements of degree $ 1, \ldots, k-1 $ with an ordering satisfying $ Y < Z $ if $ \deg Y < \deg Z  $. Then by Hall's construction $ [Y,Z] $ is a basis element of degree $ k $ if and only if $ Y $ and $ Z $ are elements of the Hall basis satisfying
	\begin{itemize}
		\item[(i)] $ Y < Z $;
		\item[(ii)]  $ \deg Y + \deg Z  = k$;
		\item[(iii)] if $ Z = [U, V] $, then $ Y \geq U $. \label{iii}
	\end{itemize}
	Assuming that $ \mathrm{ad}^k_X Y_i \in \h  $ for all basis elements $ Y_i $ of $ \partial W $ and $ k\in \N$,
	we shall show, by induction on $m$, that $ V_2 \oplus \cdots \oplus V_m  \sus \h $, for $m\geq2$. Clearly, we have that $ V_2 \sus \h$. Assume then that $ V_2 \oplus \cdots \oplus V_{k-1}  \sus \h$ and take an element $\widetilde{Y}= [Y,Z] $ of the Hall basis of degree $ k $. Recall that by (i) we have $ Y<Z $. If $ 2 \leq \deg Y, \deg Z \leq k-2 $, then $ \widetilde{Y}\in \h  $ by the induction hypothesis. Assume instead that $ \deg Y = 1 $ and $ \deg Z = k-1 $. Thus either $ Y = X $ or $ Y \in \partial W $ by construction of the  basis. If $ Y\in \partial W $, we have again that $ \widetilde{Y}\in\h$ since $ Z \in V_{k-1}\sus \h $ by the induction hypothesis.
	
	Finally, suppose that $ \widetilde{Y} = [X,Z]$. Since $ \deg Z >1$, there exist some $ U,V $ with degrees less than $ k-1 $ such that $ Z = [U, V] $. By (iii) then $ X \geq U $. Since the ordering for the basis is chosen such that $ X $ is the minimal element, this implies that $ U = X$ and $ \widetilde{Y} = [X,[X,V]] $. Similarly, since $ V $ is a degree $ k-2 $ element of the Hall basis, by (iii) we have again that $ V = [X,\widetilde{V}] $ for some $ \widetilde{V} $. Repeating this argument gives us finally that $ \widetilde{Y} = \mathrm{ad}^{k-1}_X Y_i $ for some $ Y_i\in \partial W $. Hence $  \widetilde{Y}\in \h  $, by assumption.
	We have shown that 
	$ [\g,\g]=V_2 \oplus \cdots \oplus V_s  \sus \h $ and the proof is complete.
%	}
\end{proof}

\section{Sufficient criteria for semigeneration}\label{sec:sufficient}
In Definition~\ref{def:diamond2} we  introduce   Carnot groups of type $ (\Diamond) $, which are a generalization of Carnot groups of type $ (\star) $.  After that we  present a proof for Theorem~\ref{thm:diamond}, which is formulated as a corollary of Theorem~\ref{t:suff:alg:cond:for:being:a:halfspace} (see Corollary~\ref{cor:dimaond_ample}). We conclude the section with Lemma~\ref{lemma:construct_diamonds}, which is useful in the construction of examples in Section~\ref{sec:low}.
\subsection{Carnot groups of type $(\Diamond)$}
	\begin{lemma}[Equivalent conditions for the definition of type $(\Diamond)$] \label{lemma:diamond_TFAE}
	For each subalgebra $ \h $ of a stratified Lie algebra $ \g $ for which 
	$ \h\cap V_1 $ has codimension one in $V_1$,
	the following are equivalent: 
	\begin{itemize}
		\item[(a)] there exists a basis $ \{X_1,\ldots,X_m\} $ of $ V_1 $ such that
		\begin{equation}\label{eq:diamond_basis}
		\ad_{X_i}^2 X_j \in\h \ja \ad_{\ad_{X_i}^kX_j}^2 (X_i) \in \h,
		\end{equation}
		for all $ i,j = 1,\ldots, m $ and $ k\geq 2 $;
		\item[(b)] there exists a basis $ \{Y_1,\ldots,Y_{m-1} \} $ of $ \h\cap V_1 $ and $ X\in V_1 \setminus \h $ such that
		\begin{equation}\label{eq:diamond_basis2}
		\ad_{X}^2Y_i \in \h, \quad\ad_{Y_i}^2 X \in \h \ja \ad_{\ad_{X}^k Y_i}^2 (X) \in \h,
		\end{equation}
		for all $ i = 1,\ldots, m-1 $ and $ k\geq 2 $. 
	\end{itemize}
	
\end{lemma}
\begin{proof}
	To show that (b) implies (a), assume that there exists a basis $ \{Y_1,\ldots, Y_{m-1} \} $ of $ \h\cap V_1 $ and $ X\in V_1\setminus \h $ satisfying \eqref{eq:diamond_basis2}. Then we shall check that the basis $ \{X,Y_1,\dots,Y_{m-1} \} $ satisfies conditions \eqref{eq:diamond_basis}. Indeed, from \eqref{eq:diamond_basis2} the only relations there are left to check are
	\[
	\ad_{\ad_{Y_i}^k Y_j}^2 (Y_i) \in \h \ja \ad_{\ad_{Y_i}^k X}^2 (Y_i) \in \h
	\]
	for all $ i,j \in \{1, \ldots,m-1\} $ and $ k\geq 2 $. The first one follows from the fact that $ Y_i, Y_j\in \h $ and that $ \h $ is a subalgebra. The second relation follows similarly, since $ Y_i \in \h $ and also $ \ad_{Y_i}^k X = \ad_{Y_i}^{k-2}(\ad_{Y_i}^2 X) \in \h $.
	
	For the direction (a) implies (b), let $ \{X_1,\ldots, X_m \} $ be a basis of $ V_1 $ for which \eqref{eq:diamond_basis} holds. Observe that $ X_l\in V_1 \setminus \h $ for some $ l\in \{1,\ldots,m \} $ since $ \h \cap V_1 $ is $ (m-1) $-dimensional. Assume, by possibly changing indexing, that $ l = m $. Since now $ V_1 = \R X_m \oplus (\h \cap V_1) $, then for each $ i \in \{1,\ldots,m-1 \} $ there exist $ a_i\in \R $ and $ Y_i \in \h \cap V_1 $ such that
	\[
	X_i = a_i X_m + Y_i.
	\]
	Therefore, for each $ i \in \{1,\ldots,m-1 \} $  we have
	\[
	Y_i = X_i - a_i X_m 
	\]
	and, consequently, $ \{Y_1,\ldots,Y_{m-1} \} $ is a basis of $ \h\cap V_1 $. We claim that this basis together with $ X_m \in V_1 \setminus \h $ satisfies condition \eqref{eq:diamond_basis2}. Indeed, for all $ k\geq 2 $ we have that
	\[
	\ad_{X_m}^k Y_i = \ad_{X_m}^{k-1}([X_m,Y_i]) = \ad_{X_m}^{k-1}([X_m,X_i-a_iX_m]) = \ad_{X_m}^k X_i \in \h,
	\]
	which proves that $ \ad_{X_m}^2 Y_i \in \h $ and $ \ad_{\ad_{X_m}^k Y_i}^2 (X_m) \in \h $ for all $ i= 1,\dots,m $. Furthermore,
	\[
	\ad_{Y_i}^2 X_m = [X_i-a_iX_m,[X_i-a_iX_m,X_m]] = \ad_{X_i}^2 X_m + a_i\, \ad_{X_m}^2X_i \in \h
	\]
	for every $ i= 1,\dots,m $, verifying the last missing condition of \eqref{eq:diamond_basis2}.
\end{proof}
Next definition is a restating of Definition \ref{def:diamond}.
\begin{definition}[Diamond type]\label{def:diamond2}
	Let $\g$ be a stratified Lie algebra. If each subalgebra $ \h $ of $ \g $ for which 
	$ \h\cap V_1 $ has codimension 1 in $V_1$ satisfies the
	equivalent conditions of Lemma~\ref{lemma:diamond_TFAE},
	then we say that $\g$ is  {\em of type $(\Diamond)$}.
	
\end{definition}

\begin{remark}
	Every type $ (\star) $ algebra, as introduced by Marchi, is of type $ (\Diamond) $. Indeed, we recall that a stratified Lie algebra is of type $ (\star) $ according to \cite{Marchi} if there exists a basis $ \{X_1,\dots,X_m\} $ of $ V_1 $ such that
	\begin{equation}\label{eq:star_basis}
		\ad_{X_i}^2X_j= 0 \quad \forall i,j = 1,\dots,m.
	\end{equation}
	Therefore, every such an algebra trivially satisfies \eqref{eq:diamond_basis} for every subalgebra $ \h $ of $ \g $.
\end{remark}

\begin{remark}\label{rmk:diamond+abelian_star}
	If $ \g $ is of type $ (\Diamond) $ and admits a step 2 subalgebra $ \h $ such that $ \h\cap V_1 $ has codimension 1 in $ V_1 $, then $ \g $ is of type $ (\star) $. Indeed, by \eqref{eq:diamond_basis} there exists a basis $ \{X_1,\dots,X_m\} $ of $ V_1 $ such that
	\[
	\ad^2_{X_i}X_j \in \h\cap V_3 = \{0\},
	\]
	for all $ i,j = 1,\dots,m $.
\end{remark}
Despite its simplicity, the following remark will prove useful when finding out if a given Lie algebra is of type $ (\star) $.
\begin{remark}\label{rmk:equiv_type_star}
	A Lie algebra $ \g $ is of type $ (\star) $ if and only if there exists a basis $ \{X_1,\dots,X_m \} $ of $ V_1 $ such that
	\[
	\ad_{X_i}^2(V_1)= 0 \quad  \forall\, i= 1,\dots,m.
	\]
	Indeed, this follows from the fact that the map $ Y \mapsto \ad^2_XY $ is linear for every $ X \in \g $.
\end{remark}
Next we prove   a result that is finer than Theorem~\ref{thm:diamond}. The latter is then obtained in  Corollary~\ref{cor:dimaond_ample} as an immediate consequence of Theorem~\ref{t:suff:alg:cond:for:being:a:halfspace}.
\begin{thm}
	\label{t:suff:alg:cond:for:being:a:halfspace}
	Let $ \g $ be a stratified Lie algebra.
	A horizontal half-space	$ W \subseteq \g$   is  semigenerating
	if and only if there exists a basis $ \{X_1,\ldots,X_m\} $ of $ V_1 $ such that, for $\s \coloneqq \Cl(\s_W)$,
	\begin{equation}\label{eq:thm_diamond}
	\ad_{X_i}^2 X_j \in \edge \s \ja \ad_{\ad_{X_i}^kX_j}^2 (X_i) \in \edge \s,
	\end{equation}
	for all $ i,j = 1,\ldots, m $ and $ k\geq 2 $. 
\end{thm}

\begin{proof}
	If $ W $ is semigenerating, then $ [\g,\g]\subset \edge{\s} $ and hence \eqref{eq:thm_diamond} is satisfied by any basis. Vice versa, we assume a basis satisfying \eqref{eq:thm_diamond} exists and plan to show that $ W $ is semigenerating.
	
	We start by noticing that $ \mathfrak e(\s) \eqqcolon \h $, as defined in \eqref{inv0}, is a Lie algebra (see Lemma~\ref{lem:CH}.2) for which $ \h\cap V_1 = \partial W $ has codimension 1 in $ V_1 $. Then, since \eqref{eq:thm_diamond} implies \eqref{eq:diamond_basis} by Lemma~\ref{lemma:diamond_TFAE} of such a subalgebra, there exists a basis $ \{Y_1,\ldots Y_{m-1} \} $ for $  \edge \s \cap V_1 $ and $ X\in V_1\setminus  \edge \s  $ that satisfy equations \eqref{eq:diamond_basis2}.  We claim that, to prove that $ W $ is semigenerating, it suffices to show that
	\begin{equation}\label{eq:suff_terms}
	\ad_{X}^kY_i \in  \edge \s \quad \forall i = 1, \ldots, m-1 \text{ and } k\geq 1.
	\end{equation} 
	Indeed, this is a consequence of Lemma~\ref{subalg}: being $ \ad_X $ linear, conditions \eqref{eq:suff_terms} would imply that $ \ad_X^kY \in \mathfrak e(\s) $ for every $ Y \in \partial W $. Then, by Lemma~\ref{subalg} we have $ [\g,\g] \sus \edge{\s}\sus \s $, which would prove that $ W $ is semigenerating.
	
	To show \eqref{eq:suff_terms}, we treat first the case $ k=1 $. Recall that, up to changing sign we have that $ X \in \mathfrak w _\s$. Since by \eqref{eq:diamond_basis2} we have $ \ad_{Y_i}^2 X \in  \edge \s  $ for all $ i = 1,\dots,m-1 $, then by Lemma~\ref{lemma:new_invariant_directions} we have that $ [X,Y_i]\in \edge{\s} $ for all $ i= 1,\dots,m-1 $. Since also $ [Y_i,Y_j] \in \edge \s$ for all $ i,j = 1,\dots,m-1 $ due to the fact that $ \edge \s$ is a Lie algebra, we deduce that
	\[
	V_2 = \mathrm{span}\{[X,Y_i],[Y_i,Y_j]\mid i,j= 1,\dots,m-1 \} \sus  \edge \s.
	\]

	The case $ k\geq 2 $ is proven by induction. The first step $ \ad_{X}^2Y_i \in  \edge \s $ is given by \eqref{eq:diamond_basis2}. Let us then assume that $ \ad_{X}^kY_i \in  \edge \s  $ for some $ k\geq 2 $. Since also $ \ad_{\ad_{X}^kY_i}^{2} (X) \in \edge \s  $ by \eqref{eq:diamond_basis2}, then by Lemma~\ref{lemma:new_invariant_directions} again we obtain
	\[
	[X,\ad_{X}^kY_i] =  \ad_{X}^{k+1} Y_i \in  \edge \s,
	\]
	which we needed to show.
\end{proof}
\begin{cor}[Theorem~\ref{thm:diamond}]\label{cor:dimaond_ample}
	Every Carnot algebra of type $( \Diamond) $ is semigenerated.
\end{cor}	
\begin{proof}
	Let $ \g $ be of type $ (\Diamond) $ and consider a horizontal half-space $ W $ in $ \g $. Denoting $ \s \coloneqq \Cl(\s_W) $, from Lemma~\ref{lem:CH} we have that $ \edge{\s} $ is a subalgebra of $ \g $ for which $ \edge{\s}\cap V_1 = \partial W $ has codimension 1 in $ V_1(\g) $. Being $ \g $ of type $ (\Diamond) $ we apply  \eqref{eq:diamond_basis}
	with $\h=  \edge{\s}$ to have that there exists a basis $ \{X_1,\dots, X_m\} $ of $ V_1(\g) $ satisfying
	\[
	\ad_{X_i}^2 X_j \in \edge \s \ja \ad_{\ad_{X_i}^kX_j}^2 (X_i) \in \edge \s,
	\]
	for all $ i,j = 1,\ldots, m $ and $ k\geq 2 $. Hence $ W $ is semigenerating by Theorem~\ref{t:suff:alg:cond:for:being:a:halfspace}. Since $ W $ was arbitrary, we conclude that $ \g $ is semigenerated.
\end{proof}

The following lemma gives a method to construct examples of algebras of type $ (\Diamond) $  by taking suitable quotients of product Lie algebras. In Example~\ref{ex:engel*engel}, we shall use Lemma~\ref{lemma:construct_diamonds} to give an example of a Lie algebra of type $ (\Diamond) $ that is not of type $ (\star) $.
\begin{lemma}\label{lemma:construct_diamonds}Given $ n\in \N $ and stratified Lie algebras $ \g_l $ for each $ l \inn $, let $ \g \coloneqq \prod_{l=1}^{n} \g_l $ with projections $ \pi_l \colon \g \to \g_l $ and bases $ \{X_1^l,\dots,X_{m_l}^l \} $. If $ \mathfrak i$ is a homogeneous ideal of $ \g $ such that
	\begin{equation}\label{eq:proj_diamond}
	\ad_{X^l_i}^kX^l_j  \in \pi_l(\mathfrak i) \quad \text{and}\quad \ad_{\ad_{X_i^l}^kX_j^l}^2 (X_i^l)  \in \pi_l(\mathfrak i),
	\end{equation}for all $ i,j \in \{1,\dots,m_l\} $, $ k\geq 2 $ and $ l \inn $, then $ \g / \mathfrak i $ is of type $ (\Diamond) $.
\end{lemma}
\begin{proof}
	Let $ \h $ be a subalgebra of $ \g/\mathfrak i $ for which $  \h\cap V_1(\g/\mathfrak i) $ has codimension 1 in $ V_1(\g/\mathfrak i) $ and, denoting by $ \pi $ the projection $ \pi \colon \g \to \g/\mathfrak i $, let $ \widetilde \h $ be a subalgebra of $ \g $ for which $ \pi(\widetilde \h) = \h $.
	Notice first that the set $ \{\pi(X_1^l),\dots,\pi(X_{m_l}^l) \}_{l=1}^n $ spans $ V_1(\g/\mathfrak i) $. Taking  \eqref{eq:diamond_basis} into account, since $ [X_i^l,X_j^k]= 0 $ whenever $ l\neq k $, to prove semigeneration of $ \g/\mathfrak i $ it is enough to check that
	\begin{equation}\label{eq:p_diamond2}
		\ad_{\pi(X^l_i)}^k\pi(X^l_j) = \pi(	\ad_{X^l_i}^kX^l_j) \in \h \ja \ad_{\ad_{\pi(X_i^l)}^k\pi(X_j^l)}^2 \pi(X_i^l) = \pi(\ad_{\ad_{X_i^l}^kX_j^l}^2 (X_i^l)) \in \h,
	\end{equation}
	for all $ i,j \in \{1,\dots,m_l\} $, $ k\geq 2 $ and $ l \inn $.
	
	To do this, let $ \langle \cdot,\cdot \rangle $ be a scalar product on $ V_1(\g) $ that makes the basis $ \{X_1^l,\dots,X_{m_l}^l \}_{l=1}^n  $ orthonormal. Let $ \nu \in V_1(\g)\setminus \{0\} $ be a vector that is orthogonal to $ \widetilde \h \cap V_1(\g) $, {i.e.}, let $ \nu \in \widetilde \h^\perp\cap V_1(\g) $. Write $ \nu $ as
	\[
	\nu = \sum_{l=1}^{n}\sum_{i=1}^{m_l}a^l_iX^l_i,\quad  a^l_i \in \R.
	\]
	Without loss of generality, assume that $ a_1^1 = 1 $. Then, for every $ l =2,\dots,n $ and $ i = 1,\dots,m_l $ we have that
	\[
	Y_i^l = X_i^l - a_i^lX_1^1 \in \widetilde \h,
	\] 
	as now $ \langle Y_i^l, \nu \rangle = 0 $. Since $ X_1^1 $ commutes with every $ \g_l $ for which $ l\in \{2,\dots,n\}  $, we immediately deduce that  
	\[
	\widetilde \h\supset \lie{\{Y_1^l,\dots,Y_{m_l}^l \}_{l=2}^n }\cap [\g,\g] = \lie{\{X_1^l,\dots,X_{m_l}^l \}_{l=2}^n }\cap [\g,\g] =  \prod_{l=2}^{n}[ \g_l,\g_l] .
	\]
	This proves \eqref{eq:p_diamond2} for all  $ i,j \in \{1,\dots,m_l\} $, $ k\geq 2 $ and $ l \in \{2,\dots,n\} $. It is then left to show that each term in \eqref{eq:proj_diamond} with $ l=1 $ is projected to $ \h$. Let $ Z $ be such a term. By  \eqref{eq:proj_diamond}, we have $ Z \in \pi_1(\mathfrak i) $. Since $ \mathfrak i $ is homogeneous and $ Z \in [\g_1,\g_1] $, there exists some $ \widetilde Z \in  \prod_{l=2}^{n} [ \g_l,\g_l]  $ such that $ Z + \widetilde Z \in \mathfrak i $. But since $  \prod_{l=2}^{n} [ \g_l,\g_l]\sus \widetilde \h $, we have that $ Z \in \widetilde \h + \mathfrak i $ and therefore $ \pi(Z)\in \h $.
\end{proof} 
%\begin{remark}
%	Actually, to get a type $ (\Diamond) $ Lie algebra, it is enough to fix a basis $ \{X_1^i,\ldots, X^j_{m_i} \} $ on each $ V_1^i $ and then consider a quotient on $ \g_{(3)} $ such that
%	\[
%	\ad_{X^i_l}^kX^i_j \in \pi_i(I)
%	\]
%	for all $ k\geq 2 $ and $ i \inn $.  In this way you are able to have a bit larger $ V_3 $ for $ \g/I $.
%\end{remark}

\section{Some results and examples in low-step algebras}
\label{sec:low}
\noindent
In the following section we collect some lemmata that are valid in Carnot algebras of step at most 4 and which will be used later in Section 5. However, these lemmata can also be useful when proving semigeneration of specific examples in low step. In the end of this section we provide two examples in step 3 that show that, on the one hand,  algebras of type $ (\Diamond) $ form a strictly larger class than  algebras of type $ (\star) $ and, on the other hand, that yet being of type $ (\Diamond) $ is not a necessary condition for semigeneration.

The following result gives, for step $\leq 4$, equivalent conditions for being a semigenerating horizontal half-space.
	\begin{lemma}
		\label{lemma:step4_equiv_halfspace}
		Let $ \g $ be a stratified Lie algebra of step at most 4. For each horizontal half-space  $ W $ in $ \g $, writing $\s = \Cl(\s_W)$,
		%		Suppose that step of $ \G $ is at most 4 and let $ E \sus \G $ be a constant normal set. 
		the following are equivalent:
		\begin{enumerate}[(i)]
			\item $ V_2 \sus \edge \s $;
			\item $ \ad_{Y}^2X \in \edge \s$ for every   $ X\in \mathfrak w_\s $ and $ Y\in \edge \s \cap V_1 $;
			\item $ \ad_{Y}^2X \in \edge \s$, for every   $ X, Y\in    V_1 $ ;
			%		\item There exists a basis $ \{X_1,\dots, X_m \} $ of $ V_1  $ such that $ \ad_{X_i}^2X_j \in  \edge \s  $ for all $ i,j = 1,\dots,m-1 $;
			\item $ V_3 \sus \edge \s$;
			\item $ W $ is semigenerating.
		\end{enumerate} 
	\end{lemma}
	\begin{proof}
		Implications (v) $ \implies $ (iv) $ \implies $ (iii) $ \implies $ (ii) are immediate. Regarding (ii) $ \implies$ (i), recall that $ V_2 $ is spanned by elements of the form $ [Y,Y'] $ and $ [Y,X] $, where $ Y,Y'\in \edge{\s} $ and $  X\in \mathfrak w_\s $. Since, by Lemma~\ref{lem:CH}.2, $ \edge{\s} $ is a Lie algebra, each term  $ [Y,Y'] \in \edge{\s}$ and, by Lemma~\ref{lemma:new_invariant_directions}, the terms $ [X,Y]$ belong to $\edge{\s}$.
		
		Let us finally prove (i) $ \implies $ (v). We claim that it is enough to show that $ \ad_X^kY \in \edge \s $ for every $ X \in V_1 $, $ Y\in \edge \s $ and $ k=1,2,3 $. Indeed, then by Lemma~\ref{subalg} we have $ [\g,\g]\sus \mathfrak e(\s) \sus \s $ and $ W $ is semigenerating. Now $ \ad_X^1Y = [X,Y] \in \edge \s $ by (i). Then, as $ X\in V_1 \sus  \mathfrak w _\s\cup(-\mathfrak w_\s) $ and $ \ad_{[X,Y]}^2X \in V_5 = \{0\} \sus \mathfrak e(\s) $, by Lemma~\ref{lemma:new_invariant_directions} we have $ [[X,Y],X] = -\ad_X^2Y \in \edge{\s} $. Similarly, $ \ad_{\ad_X^2Y}^2X = 0 $ and therefore $ [\ad_X^2Y,X] = \ad_X^3Y \in \edge{\s}$. So (v) follows.
	\end{proof}
\begin{remark}
Let us observe what happens to condition $(\Diamond)$ in low step.
	Given $ k\geq 2 $, the vector $\ad_{\ad_{X}^k Y}^2 (X)$ is in $ V_{2k + 3} $. Hence, if $ \g $ is of step $ s $, it is enough to require the conditions \eqref{eq:diamond_basis} or \eqref{eq:diamond_basis2} for all $ k \leq (s-3)/2 $. In particular, if $ s \leq 6 $, then a horizontal half-space $ W $ of $ \g $ is semigenerated if there exists a basis $ \{X_1,\ldots,X_m\} $ of $ V_1 $ such that
	\[
	\ad_{X_i}^2 X_j \in \lie{\partial W}
	\]
	for all $ i,j = 1,\ldots, m $. Here, we denote by $\lie{\partial W}$ the Lie subalgebra of $\g$ generated by the subset ${\partial W}$. %, since $ \lie{\partial W} $ is always contained in 
\end{remark}
Similarly to Lemma~\ref{lemma:new_invariant_directions}, the following lemma gives (in step at most 4) a method to deduce new directions that are contained in the edge of a semigroup generated by a horizontal half-space. Lemma~\ref{lemma:step3_ideal(invV2)_is_inv} below will be used in Example~\ref{ex:tame_not_necessary} and again in the proof of Proposition~\ref{prop:minimal_nonCNP_step3_isEngeln}.

	\begin{lemma}\label{lemma:step3_ideal(invV2)_is_inv}
	Let $ \g $ be a stratified Lie algebra of step  at most  4  and let $ W $ be a horizontal half-space in $ \g $. Let $\s \coloneqq \Cl( \s_W) $. If $ Z \in V_2\cap \edge \s $, then $ \Ideal_{\g}(Z) \sus \edge \s  $.
	\end{lemma}
	\begin{proof}
%Let $\s= \bar \s_W $.
Observe that $ V_1 = \R X \oplus \partial W $ for some $ X \in W\subseteq \mathfrak w_\s $.
The ideal $ \mathfrak i \coloneqq\Ideal_{\g}(Z)$ is graded and, recalling that $Z\in V_2$, we have that  its layers are
$$V_1(  \mathfrak i ) = \{0\}
,\qquad 
V_2(  \mathfrak i ) = \R Z,\qquad 
V_3(  \mathfrak i ) = 
{\rm span}\{ [X,Z] , [\partial W ,Z]    \}  
.
$$ $$%\qquad 
V_4(  \mathfrak i ) = 
{\rm span}\{ [X,[X,Z]] , [X,[\partial W ,Z]],  [\partial W ,[X,Z] ],  [\partial W , [\partial W ,Z]  ]  \} .$$
We plan to show that  $ \Ideal_{\g}(Z) \sus \edge \s $, where we 
recall that $ \edge \s $ is a Lie algebra by Lemma~\ref{lem:CH}. On the one hand, by assumption, we have that $Z\in  \edge \s $, so from $ \partial W \subseteq   \edge \s $ we get that  $ [\partial W,Z]\in \edge \s $. 
On the other hand, with the aim of   applying  Lemma~\ref{lemma:new_invariant_directions}, we observe that, since $Z\in V_2$, we have    
 $ \ad^2_Z X \in V_5 = \{0\} \in \edge{\s}$, and hence we also have $[X,Z] \in \edge{\s} $. Hence $V_3(  \mathfrak i )  \sus \edge \s $.

We also check that  $V_4(  \mathfrak i )  \sus \edge \s $. Since 
$ \edge \s $ is closed under bracket, we immediately have that $ [\partial W ,[X,Z] ],  [\partial W , [\partial W ,Z]\sus \edge \s $. 
Regarding $[X,[X,Z]] , [X,[\partial W ,Z]], $ we repeat the previous part of the argument of this proof with $Z'\in \{  [X,Z]\}\cup [\partial W ,Z] $. Indeed, we have that 
$ \ad^2_{Z'} X  = 0$ and $Z' \in \edge \s $. Hence, by  Lemma~\ref{lemma:new_invariant_directions} we also have $[X,Z'] \in \edge{\s} $. 
\end{proof}
	
%	\begin{lemma}\label{lemma:minimal_step2subalg_CNP}
%		If $ \g $ is a minimal Lie algebra of step 3 that has a stratified subalgebra of step $ 2 $, then $ \g $ is   semigenerated.
%	\end{lemma}

Next we prove that having a sufficiently large semigenerated subalgebra implies semigeneration. We shall exploit this fact later in the proof of Proposition~\ref{prop:minimal_nonCNP_step3_isEngeln}.\\

	\begin{lemma}\label{lemma:CNPstep3_subalg_implies_CNP}
		Let $ \g $ be a  stratified Lie algebra of step at most 4. If $ \g $ has a   semigenerated proper subalgebra $ \h $ such that $ V_3(\g)\sus \h $, then $ \g $ is   semigenerated.
	\end{lemma}
	\begin{proof}
	Let $ W $ be a horizontal half-space and let us show that 
	it is semigenerating.
%	$ S_W \sus G $ is a half-space. 
	Set $ H\coloneqq  V_1(\h) $. If $ H\subset \partial W $, then 
	\[
	V_3(\g) \sus \h\sus \lie{\partial W}\sus \pedge {\s_W}.
	\]
	Consequently, by Lemma~\ref{lemma:step4_equiv_halfspace} we deduce that  $ W $   is semigenerating. We then assume that  $ H \nsubseteq \partial W$. Observe that 
	%$ H= (H\cap \partial W) \oplus \R X $
	%and hence 
	$ \widetilde W \coloneqq H\cap W  $ is a horizontal half-space in $ H $.
	Hence, since by assumption  $\h$   is semigenerated, $ \widetilde W$ is semigenerating within $\h$. In particular, denoting by 
	$\bar \s_{\widetilde W}^\h$  the closure of the (log of the) semigroup generated by $ \widetilde W$   within $\h$, we have
	that $V_3(\h) \sus  \bar \s_{\widetilde W}^\h$.
%	 Then $ \exp(V_3(\h)) \sus S_{\widetilde W} $ since  $ \h $ is assumed to be ample, and therefore
Since $\h$ is assumed to contain the third layer of $\g$, we get the inclusions
	\[
	\ V_3(\g) \sus V_3(\h) \subset \bar\s_{\widetilde W}^\h \sus \bar\s_W,
	\]
	where the last containment is a consequence of the inclusions $\h\subset \g$ and $ \widetilde{W}\subset W $. Since $ V_3(\g) $ is a vector subspace of $ \g $, then by definition of $ \pedge{\bar\s_W} $ we have $ V_3(\g) \subset \pedge{\bar\s_W} $.
	Hence  $  W $ is semigenerating again by Lemma~\ref{lemma:step4_equiv_halfspace}.
\end{proof}

We remark that, actually, the above Lemma~\ref{lemma:CNPstep3_subalg_implies_CNP} has the following analogue in algebras of arbitrary step: if $ \h $ is a semigenerated subalgebra of $ \g $ and there exists a basis $ \{X_1,\dots,X_m\} $ of $ V_1(\g) $ such that the Diamond-terms \eqref{eq:diamond_basis} are in $ \h $, then $ \g $ is semigenerated. The  proof is the same, but in the final step one needs to use  Theorem~\ref{t:suff:alg:cond:for:being:a:halfspace} instead of Lemma~\ref{lemma:step4_equiv_halfspace}.

\begin{corollary}[of Proposition~\ref{prop:ample_quotient_small_center}]
	\label{lemma:minimal_nonCNP_quotient_step3}
	Let $ \g $ be a
	stratified Lie algebra of step 3.
	If $\g$ is not   semigenerated, then there exists a quotient algebra of $ \g $ that is
	trimmed and not semigenerated.
\end{corollary}
\begin{proof}
	By Proposition~\ref{prop:ample_quotient_small_center}, there exists a quotient algebra $ \hat \g $ of $ \g $ for which $ \Center(\hat \g)\cap V_1(\hat\g) =\Center(\hat \g)\cap V_2(\hat\g)=\{0\}$ and $\dim (\Center(\hat \g)\cap V_3(\hat\g)) \leq 1 $. Since the center of a stratified Lie algebra is non-trivial and homogeneous (see \eqref{center:homog}), we deduce that $ \dim \Center(\hat\g) =\dim \Center(\hat \g)\cap V_3(\hat\g) = 1  $, proving that $ \hat \g$ is trimmed.
 %	By Proposition~\ref{prop:ample_quotient_small_center}, way may assume that $ \dim \Center(\g)\cap V_j( \g) \leq1 $ for $ j=2,3 $. Since $ \g $ is of step 3, we then have $ \dim \Center(\g)\cap V_3( \g) =1 $. Let now $ W $ be a non-ample horizontal half-space of $  \g $ and take $ X \in \Center(\g)\cap V_2( \g) $. Then $ X \in V_2\cap \edge{\Cl(\s_W)) \sus \Cl(\s_W}$ by \eqref{l:V2:cap:Z(g):univ:invariant}. Therefore, by Lemma~\ref{lemma:quotient_of_cones_by_invariants}, also $\hat\g\coloneqq  \g/\R X$ is non-ample. Moreover, we have that $ \dim \Center(\hat\g)\cap V_2( \hat\g) = \{0\} $. Observe then that $ \Center(\hat\g)\cap V_1( \hat\g) $ is an abelian Carnot subalgebra of $ \hat \g $, so we may write
%	\[
%	\hat \g = \hat \g/(\Center(\hat\g)\cap V_1( \hat\g) )\times (\Center(\hat\g)\cap V_1( \hat\g) ).
%	\]
%	Since $ \Center(\hat\g)\cap V_1( \hat\g) $ is ample and $ \hat \g $ is not ample, we deduce by Remark \note{TODO: add ample product remark} that  $\tilde \g \coloneqq \hat \g/(\Center(\hat\g)\cap V_1( \hat\g) ) $ is not ample either. Hence we found a non-ample quotient algebra $ \tilde \g$ of $ \g $, for which $ \Center(\tilde\g)\cap V_1( \tilde\g) =  \Center(\tilde\g)\cap V_2( \tilde\g) = \{0\}$ and $ \dim \Center(\tilde \g)\cap V_3(\tilde \g) =1 $. Consequently, $ \dim \Center(\tilde \g) = 1 $ and $ \tilde \g $ is minimal.
\end{proof}

In the rest of this section we provide some examples. We first show a 7-dimensional Lie algebra of step 3 that is of type $ (\Diamond) $ but that is not of type $ (\star) $, see Example~\ref{ex:engel*engel}. Then we provide a 6-dimensional Lie algebra that is semigenerated but not of type $ (\Diamond) $, see Example~\ref{ex:tame_not_necessary}.
%Determining whether a given Lie algebra is of type $ (\star) $ is in principle straightforward but computationally challenging. Namely, finding out existence of a basis for $ V_1 $ which satisfies conditions \eqref{eq:star_basis} is equivalent to deciding existence of real solutions for a system of polynomial equations. In some cases, like it happens for Examples~\ref{ex:engel*engel} and~\ref{ex:tame_not_necessary}, the Gr\"obner basis for the corresponding set of polynomials (which we compute using the G\"obner basis algorithm of Mathematica) contains $ 1 $, which allows us to deduce the non-existence of a type-$ (\star) $ basis. Determining whether a Lie algebra is of type $ (\Diamond) $ is even harder.
\begin{example}\label{ex:engel*engel}
	Let $ \h_1,\h_2 $ be two copies of the four-dimensional Engel algebra $ \mathbb{En}^1 $ and consider their product Lie algebra $ \h_1\times \h_2 $. Denoting by $ Z_1 $ and $ Z_2 $ the generators of $ V_3(\h_1) $ and $ V_3(\h_2) $, respectively, and identifying $ \h_1 $ and $ \h_2 $ with the respective subalgebras of $ \h_1\times \h_2 $, we have that $ V_3(\h_1\times \h_2) = \spann{Z_1,Z_2} $. Then $ \R(Z_1-Z_2) $ is an ideal of $ \h_1\times\h_2 $ and the quotient algebra
	\[
	\g \coloneqq (\h_1 \times \h_2) / \R(Z_1-Z_2)
	\]
	is a 7-dimensional (trimmed) stratified Lie algebra of step 3 (which in the Gong's classification \cite[p.\ 57]{Gong_Thesis}
	is denoted by  (137A)). We claim that $ \g $ is of type $ (\Diamond) $ but it is not of type $ (\star) $. Indeed, the fact that $ \g $ is of type $ (\Diamond) $ follows immediately from Lemma~\ref{lemma:construct_diamonds} with $ \mathfrak i = \R(Z_1-Z_2) $, as now $ \pi_l(\mathfrak i)= V_3(\h_l) $ for both $ l=1 $ and $ l=2 $. 
	%The algorithm to verify that $ \g $ is not of type $ (\star) $ is described in Appendix~\ref{sec:appendix}.
	
	We argue next that $ \g $ is not of type $ (\star) $. Let $\{ X_1,\dots,X_7\} $ be a basis of $ \g $ %such that $ \{X_1,X_2\} $ and $ \{X_3,X_4\} $ are the Engel bases of $ V_1(\h_1) $ and $ V_1(\h_2) $, respectively, in which case. 
	for which $ \{X_1,\dots,X_4\} $ is a basis of $ V_1(\g) $ and the only nonzero brackets are $ [X_1,X_2] = X_5, [X_3,X_4]= X_6 $ and $ [X_1,[X_1,X_2]] = [X_3,[X_3,X_4]]= X_7 $, as presented in the diagram below. Then for a vector $ Y \coloneqq \sum_{i=1}^4a_iX_i \in V_1$ we have, for instance, that
	\[
	\ad^2_Y(X_2) = a_1^2X_7.
	\] In particular,  if $ Y $ is such that $ \ad^2_Y(V_1) = 0$, then $ a_1 = 0 $. Consequently, the set of vectors $ Y\in V_1 $ satisfying $ \ad^2_Y(V_1) = 0$ is contained in a 3-dimensional subspace of $ V_1(\g) $. We conclude by Remark~\ref{rmk:equiv_type_star} that $ \g $ is not of type $ (\star) $.
	 
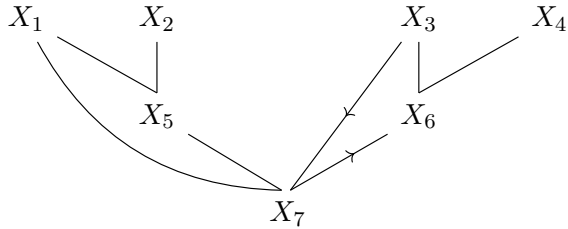
\begin{center}
	\begin{tikzcd}
		X_1 \arrow[rd, no head, end anchor = north] \arrow[rrdd, no head, bend right,  end anchor = {[xshift=-0.5ex]north}] & X_2                                        &   & X_3 \arrow[d, no head, end anchor = north] \arrow[ldd, no head, ->-=.5,end anchor = {[xshift=0.2ex]north}] & X_4 \\
		& X_5 \arrow[u, no head, start anchor = north] \arrow[rd, no head, end anchor = {[xshift=-0.5ex]north}] &   & X_6 \arrow[ru, no head, start anchor = north] \arrow[ld, no head, -<-=.4, end anchor = {[xshift=0.2ex]north}] &     \\
		&                                            & X_7 &                                             &    
	\end{tikzcd}
\end{center}
	
\end{example}
\begin{example}\label{ex:tame_not_necessary}
	Let $ \g $ be the 6-dimensional step-3 Lie algebra ($ N_{6,2,6} $ in \cite[p.\ 33]{Gong_Thesis}), where the only non-trivial brackets are given by
	\[
	[X_1,X_2]=X_4,\; [X_1,X_3]=X_5, \;[X_1,X_4]=[X_3,X_5]=X_6.
	\]
	The Lie brackets can be described by the following diagram:
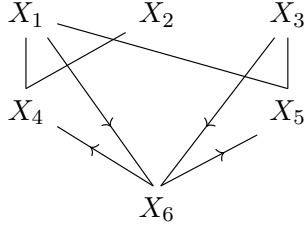
\begin{center}
 
	\begin{tikzcd}[end anchor=north]
	 X_1\arrow[d, no head] \arrow[drr, no head] \arrow[ddr, no head, ->-=.6,end anchor={[xshift=-2.8ex]north east}]& X_2\arrow[dl,no head] & X_3 \arrow[d,no head]\arrow[ddl, no head, ->-=.6, end anchor={[xshift=-2.2ex]north east}] \\
	X_4\arrow[dr, no head, -<-=.4, end anchor={[xshift=-2.8ex]north east}]& & X_5\arrow[dl, no head, -<-=.4, end anchor={[xshift=-2.2ex]north east}] \\
	 &X_6& 
	\end{tikzcd}

\end{center} 
	This Carnot algebra $ \g $ is semigenerated but it is not of type $ (\Diamond) $. Indeed, to prove that it is semigenerated, let $ W \sus \g $ be a horizontal half-space and let us show that $ V_3 = \R X_6 \sus \edge{\s} $, where $ \s \coloneqq \Cl(\s_W) $. This would show that $ W $ is semigenerating by Lemma~\ref{lemma:step4_equiv_halfspace}.
	
	Suppose first that $ X_1\notin \partial W$. The rank of $\g$ is 3, so $ \partial W $ has dimension 2.
	%	\[
	%	X = X_1 + aX_2 + bX_3
	%	\]
	%	for some $ a,b\in \R $. Hence vectors 
	Then there exist some $ a,b\in \R $ such that $ Y_2 \coloneqq X_2 - aX_1 $ and $ Y_3 \coloneqq X_3 - b X_1 $ form a basis for $\partial W $. Since $ \partial W \sus \edge{\s} $ and $ \edge{\s} $ is a Lie algebra by Lemma~\ref{lem:CH}, we obtain
	\[
	[Y_2,Y_3] = [X_2-aX_1,X_3-bX_1] = bX_4-aX_5 \in \edge{\s}.
	\]
	If $ a\neq0 $ or $ b\neq 0 $, we get
	\[
	V_3 \sus \Ideal([Y_2,Y_3]) \sus \edge\s,
	\]
	where the last inclusion comes from Lemma~\ref{lemma:step3_ideal(invV2)_is_inv}. If instead 
%	and further
%	\begin{align*}
%	[Y_2,[Y_2,Y_3]] &= [X_2-aX_1,bX_4-aX_5 ] =  -abX_6 \in \edge{\s} \\
%	[Y_3,[Y_2,Y_3]] &= [X_3-bX_1,bX_4-aX_5 ] = -(a+b^2)X_6 \in \edge{\s}.
%	\end{align*}
%	Hence, if $ ab \neq 0 $  or $ a+b^2 \neq 0 $ we get that $ V_3 = \R X_6 \sus \edge{\s}$, and $ W $ is ample.
%	
%	If instead we have $ ab = a+b^2 = 0 $, then
	 $ a = b = 0 $, then $ X_2 = Y_2 \in \edge{\s} $. Since $ \ad_{X_2}^2 X_1 = 0 $, by Lemma~\ref{lemma:new_invariant_directions} we get that $ [X_1,X_2] = X_4 \in  \edge{\s} $. Again, since $ V_3 \sus \Ideal(X_4) $, by Lemma~\ref{lemma:step3_ideal(invV2)_is_inv} we get that $ V_3 \sus \edge{\s} $.
	
	The cases $ X_2\notin \partial W $ and $ X_3\notin \partial W $ are easier: if $ X_2\notin \partial W $ we find, like above, some $ a,b\in \R $ such that $ \partial W = \spann{X_1-aX_2,X_3-bX_2} $. It then suffices to notice that $ X_6 \in \lie {X_1-aX_2,X_3-bX_2} \sus \edge{\s} $, for all choices of $ a,b \in \R $. Similarly, the case $ X_3 \notin \partial W $ follows from the fact that  $ X_6 \in \lie {X_1-aX_3,X_2-bX_3} $ for all $ a,b \in \R $. We conclude that $  \g $ is semigenerated.
	
	Finally, to justify that $ \g $ is not of type $ (\Diamond) $, observe that the span of $X_2$ and $X_3 $ is an abelian stratified subalgebra of $ \g $. If $ \g $ were type of $ (\Diamond) $, then by Remark~\ref{rmk:diamond+abelian_star} it would be of type $ (\star) $. However, similarly to Example~\ref{ex:engel*engel}, we have for an arbitrary element $ Y =a_1X_1+a_2X_2+a_3X_3 \in V_1 $ that
	\[
	\ad^2_Y(X_2) = a_1^2X_6.
	\]
	Hence vectors $ Y \in V_1 $ for which $ \ad^2_Y(V_1) = 0 $ must satisfy $ a_1=0 $, which proves the non-existence of type $ (\star) $-basis by Remark~\ref{rmk:equiv_type_star}.
	% as described in Appendix~\ref{sec:appendix}, one can verify that $ \g $ does not admit a type-$ (\star) $ basis.
\end{example} 
%%%%%%%%%%%%%%%%%%%%%%%%%%%%%%%%%%%%%%%%%%%%%%%%

	\section{Engel-type algebras}\label{sec:Engel}
In the rest of the paper we concentrate on a family of Carnot algebras that we call of {Engel type}. These algebras can be constructed through an iterative process from the classical 4-dimensional Engel algebra. Similarly to the Engel algebra, every Engel-type algebra is trimmed and non-ample, as we shall show in Propositions~\ref{rmk:engels_are_minimal} and~\ref{prop:Engels:non-CNP}. A more subtle result is that, at least in step 3, the Engel-type algebras are the only Carnot algebras with these properties. For this last part, see Proposition~\ref{prop:minimal_nonCNP_step3_isEngeln}.
The proof of Theorem~\ref{t:nonCNP:char:step6} will then be straightforward. 
	\subsection{Definition and properties}
	\begin{defn}[Engel-type algebra $ \mathbb{En}^n $]
		\label{def:nthEngel}
		For each $ n\in \N $, we denote by $ \mathbb{En}^n $ and  call it the \emph{$n$-th Engel-type algebra} the $ 2(n+1) $-dimensional Lie algebra (of step 3 and rank $ n+1 $) with basis $ \{X,Y_i,T_i,Z\}_{i=1}^n $, where the only non-trivial brackets are given by
		\begin{equation}\label{eq:Engel_brackets}
		[Y_i,X]=T_i \quad \text{and} \quad  [Y_i,T_i] = Z \quad \forall i \inn.
		\end{equation}
	\end{defn}
	 The first two  Engel-type algebras are the following. The first is $ \mathbb{En}^{1} $, and it is commonly known as Engel algebra, see \cite{Bellettini-LeDonne}, 
	 and we represent it by the diagram below.
	 
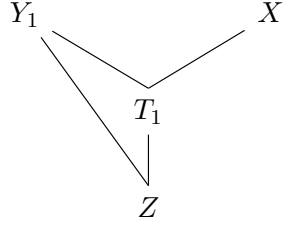
\begin{center}
    \begin{tikzcd}[end anchor=north]
    Y_1\ar[dr, no head]\ar[ddr, no head] & &X\ar[dl, no head]\\
    &T_1\ar[d, no head] & \\
    & Z &
    \end{tikzcd}
\end{center}
	 The second Engel-type algebra $ \mathbb{En}^{2} $  is the six-dimensional algebra $ N_{6,3,1a} $ in \cite[p.\ 135]{Gong_Thesis} and its diagram is presented below.
%	 \note{The diagram is not correct: due to Gong we should have $ [X_3,X_5]=[X_2,X_4]=X_6 $}

	
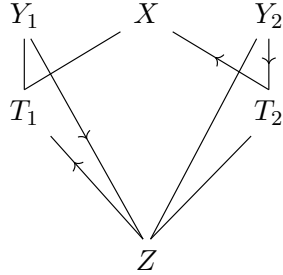
\begin{center}
		%[d, no head, -<-=.5]
			\begin{tikzcd}[end anchor=north]
	Y_1\ar[d, no head]\ar[dddr, no head,->-=.5,  end anchor={[xshift=-2.2ex]north east},start anchor={[xshift=-1.7ex]south east}] &X\ar[dr, no head,  -<-=.5]\ar[no head, dl,]  &Y_2\ar[no head, d, ->-=.5]\ar[dddl, no head,  end anchor={[xshift=-1.6ex]north east}, start anchor={[xshift=-3.2ex]south east}]\\
	 T_1\ar[ddr, no head, -<-=.3, end anchor={[xshift=-2.2ex]north east}]& & T_2\ar[ddl, no head, end anchor={[xshift=-1.6ex]north east}]\\
	& & \\
	 &Z & 
	\end{tikzcd}
%where one should read the arrows...		
	\end{center}
	Each  Engel-type algebra is a Lie algebra (see the simple verification in Remark~\ref{rmk:Engel_is_Lie}) and admits a step-3 stratification:
	\begin{eqnarray*} V_1(\mathbb{En}^n) &\coloneqq& \mathrm{span}\{ X, Y_1,\dots,Y_n\},\\
		V_2(\mathbb{En}^n) &\coloneqq& \mathrm{span}\{ T_1,\dots,T_n\},\\
		V_3(\mathbb{En}^n) &\coloneqq& \mathrm{span}\{ Z\}.
	\end{eqnarray*}

	We shall also give another equivalent definition for $ \mathbb{En}^n $  in Proposition~\ref{prop:char_of_engels}. We first list some properties of such Lie algebras.

	\begin{remark}\label{rmk:Engel_is_Lie}
		Each  Engel-type algebra given by the brackets  \eqref{eq:Engel_brackets} is indeed a Lie algebra. 
		Namely, let us verify that the Jacobi identity is satisfied. Since the  basis has a natural stratification, it is enough to check the identity for triples in the set
		$$\{X, Y_1, \ldots, Y_n\}.$$
		Hence, we just consider the case $X, Y_i, Y_j$ or $Y_i, Y_j, Y_k$. In the first case, we have 
		\begin{align*}
			[X, [Y_i, Y_j]] + [Y_i, [Y_j,X]] + [ Y_j,[X, Y_i]]  &= 0  + [Y_i, T_j, ] + [ Y_j,-T_i] \\
			&= \delta_{ij} Z - \delta_{ij} Z = 0.
		\end{align*}
		In the second case, we have
		\[
		[Y_i, [Y_j,Y_k]] + [ Y_j,[Y_k, Y_i]]  +[Y_k, [Y_i, Y_j]] = 0+0+0= 0.
		\]
	\end{remark}

	\begin{lemma}[Properties of Engel-type algebras]
		\label{lemma:properties_of_engels}
		The $n$-th Engel-type algebra $ \mathbb{En}^n $  with a basis satisfying \eqref{eq:Engel_brackets} has the following properties.
		\begin{enumerate}[(i)]
			\item If $ n \geq 2 $, then $  \mathrm{span}\{Y_1,\dots,Y_n\}$ is the unique abelian $ n $-dimensional subspace of $ V_1(\mathbb{En}^n) $;
			\item the line $ \R X $ is the unique horizontal line satisfying $ [\R X,V_2]=\{0\} $;
			\item for every nonzero $ Y \in \mathrm{span}\{Y_1,\dots,Y_n\} $, we have
			\[
			\ad_{Y}^2(V_1) = \R Z.
			\]
		%	{\color{violet}
		%	\item If $ Y, Y' \in V_1 $ are linearly independent and $ \spann{Y,Y'}\nsubseteq\spann{Y_1,\dots,Y_n} $, then $ \lie{Y,Y'} $ has step 3.}
		\end{enumerate}
	\end{lemma}
	\begin{proof} 
		\begin{enumerate}[(i)]
			\item
			Obviously, the space $ \mathrm{span}(Y_1,\dots,Y_n)$ is an abelian $n$-space. Vice versa, 
				let $ H $ be an $ n $-dimensional subspace of $  V_1(\mathbb{En}^n) $ such that $ H \neq  \mathrm{span}(Y_1,\dots,Y_n) $. Then there exists $ \nu \in H $ of the form
			\[
			\nu \coloneqq X + \sum_{i=1}^{n} a_i Y_i, \quad \text{ with } a_i \in \R \; \text{ for }  \; i= 1,\dots,n,
			\]
			and a nonzero $ Y \in H \cap \mathrm{span}(Y_1,\dots,Y_n) $. Writing $Y = \sum_{i=1}^{n} b_iY_i $ for some $ b_i \in \R $, we obtain 
			\[
			[ Y ,\nu] = \sum_i b_i T_i \neq 0,
			\]
			proving that $ H $ is nonabelian.
			\item Let now
			\[
			\nu \coloneqq a X + \sum_{i=1}^{n} a_i Y_i, \quad a,a_i \in \R \; \forall \; i= 1,\dots,n,
			\]
			where $ a_k \neq 0 $ for some $ k \inn $. Then
			\[
			[\nu,T_k] = a_kZ \neq 0,
			\]
			which shows that $ [\nu,V_2] \neq 0$ if $ \nu \notin \R X $.
			\item Let again $ Y = \sum_{i=1}^{n} b_iY_i  \ $ for some real numbers $ b_i  $ not all identically zero. Since
			\[
			\ad_{Y}^2(\R X) = \R\, \ad_{Y}^2X = \R \sum_{i=1}^{n} b_i^2Z = \R Z
			\]
			and also $ V_3(\mathbb{En}^{n}) = \R Z$, we get
			\[
			\R Z = \ad_{Y}^2(\R X) \sus \ad_{Y}^2(V_1) \sus  \R Z.
			\qedhere\]
%			\item Let first $ Y = X + \sum_{i=1}^{n} b_i Y_i $ and $ Y' =  \sum_{i=1}^{n} c_i Y_i $ be such that $ c_i \neq 0 $ for some $ i $. Then
%			\[
%			\ad_{Y'}^2Y =  \sum_{i=1}^{n}c_i^2Z \neq 0.
%			\]
%			Otherwise we may assume that $ Y' $ is of the form $ Y' =  X + \sum_{i=1}^{n} c_i Y_i $. Then 
%			\[
%			\ad_{Y'}^2Y = \sum_{i=1}^{n}c_i(c_i-b_i)Z \ja \ad_{Y}^2Y' =  \sum_{i=1}^{n}b_i(b_i-c_i)Z, 
%			\]
%			and so
%			\[
%			\ad_{Y'}^2Y+\ad_{Y}^2Y' =  \sum_{i=1}^{n}(c_i-b_i)^2Z \neq 0,
%			\]
%			since $ Y $ and $ Y' $ are linearly independent.
		\end{enumerate}
	\end{proof}

%\begin{remark}\label{rmk:automs_of_Engels2}
%	Fix a scalar product $ \langle \cdot,\cdot \rangle $ on $ V_1(\mathbb{En}^{n}) $, $ n\geq 2 $, that makes the basis in Definition~\ref{def:nthEngel} orthogonal. Then a linear map $ \Phi $ on $ V_1( \mathbb{En}^n) $ induces a Lie algebra automorphism of $ \g $ if and only if $ \Phi $ is an orthogonal transformation fixing the $ X $-axis.
%\end{remark}
%\begin{proof}
%	By Lemma~\ref{lemma:properties_of_engels}.(i) and (ii), any $ \Phi \in \mathrm{Aut}(\mathbb{En}^n) $ must fix the subspaces $ \spann{Y_1,\dots,Y_n} $ and $ \R X$. Moreover, notice that every  $ Y\in \spann{Y_1,\dots,Y_n} $ satisfies 
%	\[
%	[Y,X] = \sum_{i=1}^{n}\langle Y, Y_i \rangle T_i \ja [Y,T_i] = \langle Y,Y_i \rangle Z.
%	\] Therefore, given  $ Y,\widetilde Y\in \spann{Y_1,\dots,Y_n} $, we have
%	\[
%	[Y,[\widetilde{Y},X]] =  \sum_{i=1}^{n}\langle \widetilde Y, Y_i \rangle \langle Y,Y_i \rangle Z = \langle Y,\widetilde{Y} \rangle Z.
%	\]
%It follows that a basis $ \{X,\widetilde{Y}_1,\dots,\widetilde{Y}_n \} $ satisfies the bracket relations of Definition~\ref{def:nthEngel} if and only if $ \widetilde{Y}_1,\dots,\widetilde{Y}_n  $ is an orthogonal basis of $ \spann{Y_1,\dots,Y_n} $.
%\end{proof}
We provide next the automorphism group of the Engel-type algebras, which will be used later in the proof of Lemma~\ref{lemma:subalgs_of_engels} where we characterize all stratified subalgebras of the Engel-type algebras.
For the automorphism group of the Engel algebra, we refer to \cite[Lemma 2.3]{Bellettini-LeDonne}.
\begin{lemma}\label{rmk:automs_of_Engels}
	Consider the basis $ \{X,{Y}_1,\dots,{Y}_n \} $ of $ V_1(\mathbb{En}^{n}) $, $ n\geq 2 $, defined in \eqref{eq:Engel_brackets}. 
	Fix a scalar product $ \langle \cdot,\cdot \rangle $ on $ V_1(\mathbb{En}^{n}) $ that makes $ \{Y_1,\dots,Y_n\} $ orthonormal. Then every linear transformation on $ V_1(\mathbb{En}^{n}) $ that in the basis $ \{X,{Y}_1,\dots,{Y}_n \} $ is given by the block matrix
	\[
	\begin{pmatrix}
	a & 0 \\ 0 & b A
	\end{pmatrix}, \quad a,b \in \R \setminus\{0\} \ja A \in O(n),
	\]
	%\text{ $ a,b \in \R \setminus\{0\} $ and $ A $ is an orthogonal $ n\times n $-matrix, }
	 induces a Lie algebra automorphism of $ \mathbb{En}^{n} $. Moreover, every automorphism of $ \mathbb{En}^{n} $ is induced by such a transformation on $ V_1(\mathbb{En}^{n}) $.
\end{lemma}
\begin{proof}
	By Lemma~\ref{lemma:properties_of_engels}.(i) and (ii), every $ \Phi \in \mathrm{Aut}(\mathbb{En}^n) $ must fix the subspaces $ \spann{Y_1,\dots,Y_n} $ and $ \R X$. Moreover, notice that any linear map $ \Phi $ on $ V_1(\mathbb{En}^{n}) $ fixing these subspaces satisfies, for some $ a\neq 0 $, the two equalities:
	\begin{align*}
		[\Phi(Y_i),\Phi(X)] &= [\Phi(Y_i), aX] = \sum_{k=1}^{n}a \langle \Phi(Y_i), Y_k \rangle T_k \ja \\
		[\Phi(Y_i),T_k] &=
		[  \sum_{\ell=1}^{n} \langle \Phi(Y_i) ,Y_\ell \rangle Y_\ell,T_k] =
		 \langle \Phi(Y_i),Y_k \rangle Z.
	\end{align*}
	 Therefore, using again that   $\{Y_j\}_j$ are orthonormal, we deduce that
	\begin{equation}\label{eq:En_automs}
		[\Phi(Y_i),	[\Phi(Y_j),\Phi(X)]] =  \sum_{k=1}^{n}a \langle \Phi(Y_j), Y_k \rangle \langle \Phi(Y_i),Y_k \rangle Z = a \langle \Phi(Y_i),\Phi(Y_j) \rangle Z. 
	\end{equation}
%	\begin{align}
%	\notag	[\Phi(Y_i),	[\Phi(Y_j),\Phi(X)]] &= \sum_{k=1}^{n}a \langle \Phi(Y_j), Y_k \rangle [\Phi(Y_i),T_k] \\
%\notag		&=  \sum_{k=1}^{n}a \langle \Phi(Y_j), Y_k \rangle \langle \Phi(Y_i),Y_k \rangle Z \\
%		 &= a \langle \Phi(Y_i),\Phi(Y_j) \rangle Z. \label{eq:En_automs}
%	\end{align}
Recall that the basis vectors of $ \mathbb{En}^{n} $ satisfy $ [Y_i,[Y_j,X]] = \delta_{ij}Z $. Therefore, the map $ \Phi $ induces a Lie algebra automorphism of $ \mathbb{En}^{n} $ if and only if there exists some $ b\neq 0 $ such that
 \begin{equation*}
 	[\Phi(Y_i),	[\Phi(Y_j),\Phi(X)]] = b\, \delta_{ij}Z, \quad\,\forall i,j \inn
 \end{equation*}
 According to \eqref{eq:En_automs}, this is equivalent to saying that the map $ \Phi $ is, up to scaling, an orthogonal transformation on $ \spann{Y_1,\dots,Y_n} $ .
\end{proof}
	For a good understanding of  the rest of this section, we stress that we say that a subalgebra of a stratified Lie algebra is stratified if it is homogeneous and stratified with respect to the induced grading.
%	when we have a homogeneous subalgebra of a stratified Lie algebra 
	\begin{lemma}\label{lemma:subalgs_of_engels}
		Let $ 1\leq k \leq n  $.
		% and consider an Engel-type algebra $ \mathbb{En}^{n} $.
		 If   $ \h $ is a stratified rank-$ k $ subalgebra of the 
		 $n$-th Engel-type algebra $ \mathbb{En}^{n} $, then either $ \h $ is abelian or it is isomorphic to $ \mathbb{En}^{k-1} $.
	\end{lemma}
\begin{proof}
	If $ n=1 $, the claim is trivially true. Assume then that $ n\geq 2 $ and let $ \{X,Y_1,\dots,Y_n \} $ be a basis of $ V_1(\mathbb{En}^{n}) $ as in \eqref{eq:Engel_brackets}.
	We start by proving the case $ k= n $. Assume first that $ \h $ is the subalgebra generated by $ \{X+aY_n,Y_1,\dots,Y_{n-1} \} $ for some $ a\in \R $. Observe that then $ \h $ is isomorphic to $ \mathbb{En}^{n-1} $ for every value of $ a\in \R $ since, for each $ i= 1,\dots,n-1 $, we have that% $ [X+aY_n,Y_i] = [X,Y_i] = T_i $ and $ [X+aY_n,T_i] = [X,T_i] = 0 $.
	\[
	[Y_i,X+aY_n] = [Y_i,X] = T_i \ja  [X+aY_n,T_i] = [X,T_i] = 0.
	\]
	Recall that, by Lemma~\ref{lemma:properties_of_engels}.i, the subspace $ \spann{Y_1,\dots,Y_n} $ is the unique abelian stratified subalgebra of $ \mathbb{En}^{n} $.	
%	Observe that, for every $ n $-dimensional subspace $ V $ of $V_1( \mathbb{En}^{n}) $ which is not equal to $ \spann{Y_1,\dots,Y_n} $, there are $ a\in \R $ and a rotation $ R $ of $ V_1(\mathbb{En}^{n}) $ around the $ X $-axis, such that $ V = A(V_a) $, where $V_a \coloneqq  \spann{X+aY_n,Y_1,\dots,Y_{n-1}} $.
	Since every $ n $-dimensional subspace of $ V_1(\mathbb{En}^{n}) $ which is not equal to $ \spann{Y_1,\dots,Y_n} $ can be realized from some subspace $ \spann{X+aY_n,Y_1,\dots,Y_{n-1}} $, $ a\in \R $, by a rotation of $ V_1(\mathbb{En}^{n}) $ around the $ X $-axis, we infer by Lemma~\ref{rmk:automs_of_Engels} that any subalgebra generated by such subspace is isomorphic to $ \mathbb{En}^{n-1} $.
	
	Regarding the case $ k<n $, fix a non-abelian stratified rank-$ k $ subalgebra $ \h_k $ of $ \mathbb{En}^{n} $. Then we find a filtration $ \h_k \subset \h_{k+1} \subset \dots \subset  \h_n \subset \mathbb{En}^{n}$ of non-abelian stratified rank-$ l $ subalgebras $ \h_l $, $ l=k+1,\dots,n $. The claim follows now by the first part of this proof.
\end{proof}
There are plenty of Lie algebras of arbitrarily large step 
whose first three layers coincide with the
  first Engel-type algebra, for example, the filiform algebras. The same phenomenon does not happen for the other Engel-type algebras.
%  \footnote{A stronger property may be true: they have no central extension.}. 
%\note{This is not exactly true (or this is a little bit inaccurately stated)} 
Since we need this latter fact in the proof of Proposition~\ref{prop:char_of_engels}, we clarify such a phenomenon in the next remark.
\begin{remark}\label{rmk:Engels_not_prolognable}
	For each $ n\geq 2 $, the Lie algebra $ \mathbb{En}^n $  cannot be `prolonged' in the following sense:  if $ \g $ is a stratified Lie algebra for which $ \g / \g^{(4)} $ is isomorphic to $ \mathbb{En}^n $ for some $ n\geq 2 $, where  $ \g^{(4)}  = V_4 \oplus \dots \oplus V_s$, then $ \g = \mathbb{En}^n $. 
	
\end{remark}
\begin{proof}
	Let $ \{X,Y_i \}_{i=1}^n  $, $ \{T_i \}_{i=1}^n  $ and $ \{Z \} $ be bases of $ V_1(\g) $, $ V_2(\g) $  and $ V_3(\g) $, respectively, satisfying the bracket relations \eqref{eq:Engel_brackets}  modulo  $\g^{(4)}  $.
	Observe that, being $ \g $ stratified, the subspace $ V_4 $ is spanned by elements of the form $ [\nu,Z] $, where $ \nu \in \{X,Y_i \}_{i=1}^n$. We need to show  that $ V_4= \{0\} $. Indeed, by the Jacobi identity, we have 
	\begin{align*}
	[X,Z] = [X,[Y_i,[Y_i,X]]] =  -[Y_i,[[Y_i,X],X]] - [[Y_i,X],[X,Y_i]] = 0,
	\end{align*} where we deduced that $ [Y_i,[[Y_i,X],X]] = 0 $, since $ [X,V_2] = \{0\} $ by Lemma~\ref{lemma:properties_of_engels}.ii. Moreover, for every $ j= 1,\dots,n $ and $ i\neq j $ (which exists since $n\geq 2$) one has
	\begin{align*}
	[Y_j,Z] = [Y_j,[Y_i,[Y_i,X]]] =  -[Y_i,[[Y_i,X],Y_j]] - [[Y_i,X],[Y_j,Y_i]] = 0 ,
	\end{align*} where we used that $ [Y_j,[Y_i,X]] = [Y_j,T_i] = 0 $ and that $ [Y_j,Y_i]=0 $. This proves that $ V_4= \{0\} $, in which case also $ \g^{(4)} = \{0\} $ and hence $\g \cong \g / \g^{(4)} \cong  \mathbb{En}^n $.
\end{proof}
	
	The Engel-type algebras have the following equivalent definition using induction.
	
	\begin{prop}[A characterization of Engel-type algebras]
		\label{prop:char_of_engels}
		Let $\g$ be a stratified Lie algebra of rank $n+1\geq4$. Then $\g$ is isomorphic to 
		$ \mathbb{En}^{n} $ if and only if $\g$ has a unique abelian stratified subalgebra of rank $n$ and every other stratified subalgebra of rank $n$ is isomorphic to 
		$ \mathbb{En}^{n-1} $. Moreover, this characterization holds for rank $ n+1=3 $ if in addition $ \dim V_3(\g) = 1 $.
	\end{prop}

		\begin{proof}
		One direction is proven in Lemmata~\ref{lemma:properties_of_engels}.i and~\ref{lemma:subalgs_of_engels}. Regarding the other direction, let $ \g $ be a rank $ n+1 $ stratified Lie algebra, with $ n+1 \geq 3 $, which has a unique abelian stratified subalgebra $ \h_0 $ of rank $ n $ and every other rank-$ n $ stratified subalgebra is isomorphic to $ \mathbb{En}^{n-1} $. 
		
		Our first aim is to show that the condition $ \dim V_3(\g) = 1 $ always holds as long as    $n+1\geq 4$.
%		Let us denote by $ \h_0 $ the unique abelian $ n $-dimensional subspace of $ V_1(\g) $. 
		We start by claiming the following property:
			\begin{equation}\label{claim:abelian_subsp}
			\text{If }n+1\geq 4 \text{ and }\mathfrak l\subset V_1(\g) \text{ is an }(n-1)\text{-dimensional abelian subspace of }V_1(\g),\text{ then }\mathfrak l \subset \h_0.
			\end{equation}
			Indeed, let $ \nu \in  V_1(\g) \setminus \mathfrak l $. On the one hand, if  $ \mathfrak l \oplus \R \nu $ is an abelian subalgebra of $ \g $, then  $  \mathfrak l \oplus \R\nu = \h_0  $ by uniqueness of $ \h_0 $ and so $ \mathfrak l \subset \h_0 $.  On the other hand, if $ \mathfrak l \oplus \R \nu $ generates a nonabelian subalgebra, then $ \mathfrak l \oplus \R \nu $ is isomorphic to $ V_1(\mathbb{En}^{n-1}) $, where $ n-1\geq 2 $. Since $ \h_0 \cap  (\mathfrak l \oplus \R\nu) $ is an abelian $ (n-1)$-dimensional subspace of $V_1( \mathbb{En}^{n-1}) $, we deduce that $ \h_0 \cap  (\mathfrak l \oplus \R\nu ) = \mathfrak l$ by uniqueness of $ (n-1) $-dimensional abelian subspaces of $V_1( \mathbb{En}^{n-1}) $ because $n-1\geq 2$ (see Lemma~\ref{lemma:properties_of_engels}.i). Then again $ \mathfrak l \subset \h_0 $ and \eqref{claim:abelian_subsp} is proven.
			
			%Our next aim is to show that $ \dim V_3(\g) = 1 $ whenever $ n+1 \geq 4 $.
			Recall that
			\[
			V_3(\g) = \mathrm{span}\{[X_1,[X_2,X_3]] \mid X_i \in V_1,\; i=1,2,3 \}
			\]
			as $ \g $ is stratified. We are going to show that vectors $ [X_1,[X_2,X_3]] $ and  $ [\widetilde X_1,[\widetilde X_2,\widetilde X_3]] $ are linearly dependent, for every choice of vectors $ X_i,\widetilde X_i \in V_1$,  $ i=1,2,3 $. So let $ X_i,\widetilde X_i \in V_1$ for $ i=1,2,3 $ be such that $ [X_1,[X_2,X_3]] $ and $ [\widetilde X_1,[\widetilde X_2,\widetilde X_3]] $ are nonzero and let $ \h_1  \supseteq \{X_1,X_2,X_3\} $ and $ \h_2 \supseteq \{\widetilde X_1,\widetilde X_2,\widetilde X_3\} $ be rank-$ n $ subalgebras of $ \g $. Since $ \h_1 $ and $ \h_2 $ have nonzero third layers, they are isomorphic to $ \mathbb{En}^{n-1} $. We may assume that $ V_1(\h_1)\neq V_1(\h_2) $, since otherwise $ [X_1,[X_2,X_3]] $ and $[\widetilde X_1,[\widetilde X_2,\widetilde X_3]] $ are linearly dependent. Hence, being it the intersection of two different hyperplanes,  the space $ V_1(\h_1)\cap V_1(\h_2) $ is a codimension 2 subspace of $ V_1(\g) $, {i.e.}, it has dimension $ n-1 $.
			 
			To prove that $\dim V_3(\g)= 1 $, we have to show that 
			$V_3(\h_1) = V_3(\h_2)$, for all such $\h_1$ and $\h_2$ as above.
			The proof for the latter fact  
			is divided into two cases depending on whether $ V_1(\h_1)\cap V_1(\h_2) $   is closed under brackets (or equivalently, it is an abelian subalgebra) or not. Assume first that $ V_1(\h_1)\cap V_1(\h_2) $ forms an abelian subalgebra. Then by claim \eqref{claim:abelian_subsp} we have that $ V_1(\h_1)\cap V_1(\h_2)\sus \h_0 $. Let us fix a basis $ \{Z_1,\dots,Z_{n-1} \} $ for $ V_1(\h_1)\cap V_1(\h_2) $ and let also $ Z_n \in\h_0 $ and $ X\in V_1(\g) $ be such that $ \{Z_1,\dots,Z_n,X \} $ is a basis of $ V_1(\g) $. Fix next $ Y_i \in V_1(\h_i)\setminus( V_1(\h_1)\cap V_1(\h_2) )$ for $ i=1,2 $ and write it in terms of this basis as
			\[
			Y_i = a_i X + \sum_{j=1}^{n}b_j^i Z_j
			\]
			for some $ a_i, b_j^i \in \R $.
			Notice that, since $\h_i$ is not abelian, we have
			 $ a_i \neq 0 $.
			Since now $ \{Y_i,Z_1,\dots,Z_{n-1} \} $ is a basis of $ V_1(\h_i) $ for $ i=1,2 $ and since $ \h_0 = \spann{Z_1,\dots,Z_n } $ is abelian, we obtain 
			\begin{align*}
			V_3(\h_1) = \ad_{Z_1}^2(V_1(\h_1)) = \ad_{Z_1}^2(\R Y_1) =  \ad_{Z_1}^2(\R X) =  \ad_{Z_1}^2(\R Y_2) = \ad_{Z_1}^2(V_1(\h_2))=V_3(\h_2),
			\end{align*}
			where the first and the last equality follow from Lemma~\ref{lemma:properties_of_engels} (iii).
			Since the third layers of $ \h_1 $ and $ \h_2 $ are one dimensional, we deduce that $ [X_1,[X_2,X_3]] $ and $[\widetilde X_1,[\widetilde X_2,\widetilde X_3]] $ are linearly dependent.
			
		 Assume instead that $ V_1(\h_1)\cap V_1(\h_2) $ is not a subalgebra. Then, by Lemma~\ref{lemma:subalgs_of_engels}, the Lie algebra generated by $ V_1(\h_1)\cap V_1(\h_2) $ is isomorphic to $ \mathbb{En}^{n-2} $. In particular, it has step 3. Exploiting again the fact $ \dim V_3(\h_1) = \dim V_3(\h_2) =1 $, we get that
			\[
			V_3(\h_1) = V_3(\lie{ V_1(\h_1)\cap V_1(\h_2)}) = 	V_3(\h_2).
			\]
			Therefore  $ [X_1,[X_2,X_3]] $ and $[\widetilde X_1,[\widetilde X_2,\widetilde X_3]]$ are linearly dependent as in the previous case. This concludes the proof for the fact that $ \dim V_3(\g) = 1 $ when $ n+1 \geq 4 $. Hence, in what follows we may assume that  $ V_3(\g) $ is one dimensional, and that $n+1\geq 3$.
		
		In the rest of this proof we are going to construct a basis of $ \g $ that satisfies the defining commutator relations \eqref{eq:Engel_brackets} of the Engel-type algebra $ \mathbb{En}^n $. Let $ \h \cong \mathbb{En}^{n-1}$ be some nonabelian rank-$ n $ subalgebra of $ \g $ and let $ \{X,Y_i,T_i,Z \}_{i=1}^{n-1} $ be a basis of $ \h $ satisfying relations \eqref{eq:Engel_brackets}. We aim to find a vector $ Y_n \in V_1 $ which together with $T_n \coloneqq [Y_n,X] $ completes  $ \{X,Y_i,T_i,Z \}_{i=1}^{n-1} $ to the defining basis of $ \mathbb{En}^{n} $. Since $ \mathbb{En}^{n} $ cannot be prolonged (see Remark~\ref{rmk:Engels_not_prolognable}), this is enough to prove that $ \g $ is isomorphic to $ \mathbb{En}^{n} $. Notice that $ \h_0\cap \h = \spann{Y_1,\dots,Y_{n-1}} $ since $ \spann{Y_1,\dots,Y_{n-1}} $ is the unique abelian subspace of $ V_1(\h) $ by Lemma~\ref{lemma:properties_of_engels} (i). Moreover, $ V_3(\g) = \R Z $ as $ V_3(\g) $ is one dimensional. Fix next $ \hat Y_n\in \h_0 \setminus \h $ and write
			\[
			Y_n \coloneqq a\left (\hat{Y}_n + \sum_{i=1}^{n-1}a_iY_i  \right ) \in \h_0,
			\]
			where $a, a_i \in \R $, $ i= 1,\dots,n-1 $, are values to be determined later. Now  whenever $ a \neq 0 $ we have that $ \spann{Y_1,\dots,Y_n} =\h_0$ is abelian and $ \{Y_1,\dots,Y_n,X\} $ is a basis of $ V_1(\g) $. To conclude the proof of the proposition, we claim that it suffices to show that
			%	It suffices to find $a, a_i \in \R $, $ i= 1,\dots,n-1 $ such that
			\begin{enumerate}[(i)]
				\item there exist $ a_i \in \R $, $ i= 1,\dots,n-1 $, such that $ [Y_n,T_i]=0 $ for all $ i= 1,\dots,n-1 $;
				\item there exists $ a\in \R $ such that $ [Y_n,[Y_n,X]]=Z $;
				\item with the above choices of $ a_i,a \in \R $, the vector $ T_n\coloneqq[Y_n,X] $ is linearly independent of $ T_1,\dots,T_{n-1} $.
			\end{enumerate}
Indeed, we stress again that, by	Remark~\ref{rmk:Engels_not_prolognable}, the Lie algebra		 $ \mathbb{En}^{n} $ cannot be prolonged and hence we indeed have that the set $ \{X,Y_i,T_i,Z \}_{i=1}^{n-1} $ is a basis of a step-3 Lie algebra isomorphic to $ \mathbb{En}^{n} $.
			
			To show (i), observe that for every $ i= 1,\dots,n-1 $,
			\[
			[Y_n,T_i] = a([\hat Y_n,T_i] + a_iZ).
			\]
			Since $ V_3(\g) $ is one dimensional, the two vectors $ [\hat Y_n,T_i]  $ and $ Z $ are linearly dependent. Hence for every $ i= 1,\dots,n-1 $ there exists $ a_i \in \R $ such that $ [Y_n,T_i] = 0 $, which proves (i).
			
			Regarding (ii), let then $ \h' \coloneqq \lie{Y_2,\dots,Y_n,X} \cong \mathbb{En}^{n-1} $. Since $ \spann{Y_2,\dots,Y_n} $ is again the unique abelian $ (n-1) $-dimensional subspace of $ V_1(\h') $, it holds
			\[
			[Y_n,[Y_n,\R X]] = [Y_n,[Y_n,V_1(\h')]].
			\]
			As $ [Y_n,[Y_n,V_1(\h')]]\neq 0 $ by Lemma~\ref{lemma:properties_of_engels} (iii) and since $ V_3(\g) = \R Z $, we may choose $ a\in \R $ such that
			\[
			[Y_n,[Y_n,X]]= Z.
			\]
			Thus (ii) is proven.
			
			Regarding (iii), it is enough to notice that from (i) we have
			\[
			[Y_n,\sum_{i=1}^{n-1}b_iT_i] = \sum_{i=1}^{n-1}b_i[Y_n,T_i] = 0,
			\]
			for every choice of $ b_i \in \R $, $ i= 1,\dots,n-1 $, whereas from (ii) we have
			\[
			[Y_n, T_n] = [Y_n,[Y_n,X]]\neq 0.
			\]
			This finishes the proof of (iii) and  of the proposition.
	 \end{proof} 
 	Next we show that each Lie algebra $ \mathbb{En}^n $ is trimmed and that a horizontal half-space $ W $ is, for $ n\geq 2 $, not ample if and only if $ \partial W $ is the abelian codimension 1 subspace of $ V_1 $. Recall our definition of trimmed   algebra  from Proposition~\ref{lemma:minimal_TFAE}.
 \begin{proposition}\label{rmk:engels_are_minimal}
 	Every Engel-type algebra is trimmed.
 \end{proposition}
 \begin{proof}
 	Let $ n\in \N $ and let $ \mathbb{En}^n $ be the Engel-type algebra with the defining basis \eqref{eq:Engel_brackets}. 
 	We shall use the third definition of trimmed Lie algebra in Proposition~\ref{lemma:minimal_TFAE}.
 	Recall, see \eqref{center:homog},  that in every step-$s$ stratified Lie algebra $\g$, we have that the center $\Center(\g)$ is graded and $V_s\subseteq \Center(\g)$.
 	Hence,	noticing that  $ \dim V_3 = 1 $, it suffices to show that 
 	$    \Center(\mathbb{En}^n ) \cap V_1 =    \Center(\mathbb{En}^n ) \cap V_2 = \{0\}$. By Lemma~\ref{lemma:properties_of_engels}.ii, we have $ [Y,V_2]\neq 0 $ for every $ Y \in V_1 \setminus \R X $. Since $ X $ is not in the center either, this proves that $    \Center(\mathbb{En}^n ) \cap V_1 =\{0\}$. To show that $   \Center(\mathbb{En}^n ) \cap V_2 = \{0\}$, let $ T \in V_2 $ be nonzero and write
 	\[
 	T \coloneqq \sum_{i=1}^{n}a_iT_i, \quad a_i \in \R \;\forall\; i= 1,\dots,n.
 	\]
 	Then $ a_i\neq 0 $ for some $ i= 1,\dots,n $ and hence
 	\[
 	[Y_i,T] = a_iZ \neq 0. \qedhere
 	\]
 \end{proof} 
 
	\begin{remark}
		A half-space $ W \sus \mathbb{En}^n $ for $ n\geq 2 $ is ample whenever $ \partial W \neq \spann{Y_1,\dots,Y_n} $.
	\end{remark}
\begin{proof}
	By Lemma~\ref{lemma:subalgs_of_engels}, we have that $  \lie{\partial W} $ is isomorphic to $ \mathbb{En}^{n-1} $. In particular, $ V_3(\lie{\partial W}) \neq \{0\}$. Since $  V_3(\lie{\partial W}) \sus V_3(\mathbb{En}^n)$ and $ V_3(\mathbb{En}^n) $ has dimension one, we deduce that $ V_3(\mathbb{En}^n) \sus  \lie{\partial W}  \sus \pedge{\s_W}$. The proof is finished by Lemma~\ref{lemma:step4_equiv_halfspace}. % as $ \lie{\partial W}  \sus \edge{\s_W}$ by Lemma~\ref{lem:CH}.2. 
\end{proof}

	\begin{prop} \label{prop:Engels:non-CNP}
		None of the
		%	$ n\in \N $, the $ n^{th} $ 
		$ \mathbb{En}^n $ is     ample. Indeed, using the defining basis \eqref{eq:Engel_brackets}, every (of the two) horizontal  half-space $W$ such that $ \partial W = \spann{ Y_1,\dots,Y_n}   $ is not ample.
	\end{prop}
	\begin{proof}
		Let $ n\geq 1 $ and let us consider the following explicit representation of the basis elements as vectors in $ \R^{2(n+1)} $:
		\begin{align*}
		Y_i &= \doo i;  \\
		X &= \doo {n+1} + \sum_{i=1}^{n} x_i \doo{n+1+i} + \dfrac{x_i^2}{2}\doo{2(n+1)}; \\
		T_i &= \doo{n+1+i} + x_i\doo{2(n+1)} ;\\
		Z &= \doo {2(n+1)},
		\end{align*}
		where $ i= 1,\dots,n $. It is readily checked that these vector fields satisfy the commutator relations given in Definition~\ref{def:nthEngel}.  
		We consider the    set $$ C \coloneqq \{x \in \R^{2(n+1)} : x_{2(n+1)}\geq 0 \} .$$
	 	We shall show that, for $W\coloneqq \R_+ X \oplus \spann{ Y_1,\dots,Y_n}$,
	 	the set $C$ contains $\Cl(S_W)$ but $\exp(V_3)$ is not contained in $C$.
	 	
	 	Regarding  $\Cl(S_W) \subset C$, we claim that it is enough to show that 
	 	\begin{equation}\label{claim:SW_subset_C}
	 		p\exp  (\R_+ Y) \subseteq C, \quad \forall\, p\in C,\quad \forall \, Y \in W.
	 	\end{equation}
	 	Indeed, assume that  \eqref{claim:SW_subset_C} holds. Since $ C $ is closed, it suffices to prove that $ S_W \sus C $. As $ 0 \in C $, then by \eqref{claim:SW_subset_C} we have that $ \exp(Y) \in C $ for all $ Y \in W $. Then, again by \eqref{claim:SW_subset_C}, for every finite collection $ Y_1,\dots,Y_k \in W $ it holds
	 	\[
	 	\exp(Y_1)\cdots\exp(Y_k) \in C.
	 	\]
	 	Therefore, we conclude by \eqref{def:SW} that
	 	\[
	 	S_W \overset{\eqref{def:SW}}{=} \bigcup_{k=1}^\infty (\exp(W))^k \sus C,
	 	\]
	 	which we needed to show.

		 To prove claim \eqref{claim:SW_subset_C}, we use the fact that the curve $t\mapsto p\exp  (t Y)$ is the flow line of the vector field $Y$ starting from $p$. Write $Y= a X + \sum_{i=1}^n b_i Y_i$ with $a\geq 0$ and $b_i\in \R$.
		 The ODE given by this vector field is
		 \begin{equation*}%\label{eq:monotone_dir}
	 		 \doo{t}  \big(p\exp  (t X) \big) = a X_{p\exp  (t X)} + \sum_{i=1}^n b_i (Y_i)_{p\exp  (t X)}.
			 % \dfrac{\big(p\exp  (t X)\big)_i^2}{2},
		 \end{equation*}
		In particular, from the above expression of the vector fields in coordinates, we have that  the $2(n+1)$-th component of $ p\exp  (t X)$ satisfies
		 \begin{equation}\label{eq:monotone_dir}
	 		 \doo{t} \big(p\exp  (t X)\big)_{2(n+1)} = \sum_{i=1}^n \dfrac{\big(p\exp  (t X)\big)_i^2}{2},
		 \end{equation}
		 which is non-negative. Notice that if $ p\in C $, then at time $t=0$ the $2(n+1)$-th component is non-negative, i.e., $\big(p\exp  (t X)\big)_{2(n+1)} = p_{2(n+1)}\geq 0$. Therefore, $p\exp  (t X) \in C$ for all $t\in \R_+$ by \eqref{eq:monotone_dir}, and so claim \eqref{claim:SW_subset_C} is proven.
		% So, being $p\in C$, then at time $t=0$ the component $\big(p\exp  (t X)\big)_{2(n+1)} = p_{2(n+1)}\geq 0$, being $p\in C$, then after a positive time it is still non-negative. Thus $p\exp  (t X) \in C$, for all $t\in \R_+$. 
%		 Similarly, we have $p\exp  (t Y_i) \in C$ for all $t\in \R$ because $C$ is $\doo i$ invariant for all $i=1, \ldots, n$.
		 
%		Consider the projection map $ P \colon \R^{2(n+1)}\to \R $, $ P(x)= x_{2(n+1)}$ and the set $ C \coloneqq \{x \in \R^{2(n+1)} : x_{2(n+1)}\geq 0 \} $. As $ C $ is a level set of a smooth function, the distributional derivative $ \nu \mathbb{1}_C $ for $ \nu \in \g $ is given by
%		\[
%		\nu \mathbb{1}_C = -\dfrac{\nu P}{\abs{\nabla P}},
%		\]
%		where $ \nabla = X + \sum_{i=1}^{n}Y_i$ is the horizontal gradient (see \cite{AKL} for reference). Since 
%		\begin{align*}
%		Y_iP &= 0;\\
%		XP &= \sum_{i=1}^n \dfrac{x_i^2}{2} \geq 0
%		\end{align*}
%		for all $ i= 1,\dots,n $, the set $ C $ has $ X $ as constant horizontal normal. On the other hand, $ C $ is not a halfspace as
%		\[
%		ZP = 1 \neq 0.
%		\]
To   see that $\exp(V_3)$ is not contained in $C$, we  observe that $V_3=
\R Z$ and $Z = \doo {2(n+1)}$. We get the conclusion since $C$ is not $ \doo {2(n+1)}$-invariant.
	\end{proof} 

%Before proceeding to the proof of Theorem ~\ref{t:nonCNP:char:step6}, we shall give a
For the time being, we do not know if the intrinsic $ C^1 $-rectifiability result for finite-perimeter sets according to \cite{fssc}  holds in non-ample Carnot groups. However, in \cite[Corollary 5.12]{DLMV}, \emph{a priori} weaker intrinsic Lipschitz rectifiability of finite-perimeter sets was shown in every Carnot group that admits a non-abnormal horizontal line. As a final remark before proceeding to the proof of Theorem ~\ref{t:nonCNP:char:step6}, we shall characterize non-abnormal horizontal lines of the $ n $-th Engel-type algebras for $ n\geq 2 $. The case $ n=1 $ is treated in \cite[Section 6]{DLMV}. As a consequence, we have that the reduced boundary of a finite-perimeter set in any Engel-type algebra is intrinsically Lipschitz rectifiable.
\begin{remark}
	Let $ n\geq 2 $ and consider $ \mathbb{En}^n $ with the basis $ \{X,Y_1,\dots,Y_n \} $ satisfying \eqref{eq:Engel_brackets}. Then, for $ \nu \in V_1\setminus\{0\}$, the line $ t\mapsto \exp(t \nu) $ is abnormal if and only if $ \nu \in \spann{Y_1,\dots,Y_n}\cup \R X $.
\end{remark}
\begin{proof}
	By \cite[Proposition 5.10]{DLMV} and the fact that $ \mathbb{En}^n $ is stratified, for a given $ \nu \in V_1\setminus\{0\}$, the curve $ t\mapsto \exp(t \nu) $ is non-abnormal if and only if
%	\begin{equation}\label{eq:adjoint}
%	\mathrm{span}\{\mathrm{ad}^k_\nu (V_1) : k =0,1,2 \} = \mathbb{En}^n
%	\end{equation}
%	\begin{equation}\label{eq:adjoint}
%\mathrm{ad}_\nu (V_1) = V_2 \ja \mathrm{ad}^2_\nu (V_1) = V_3.
%\end{equation}
%\begin{enumerate}[(i)]
%	\item $ \mathrm{ad}_\nu (V_1) = V_2  $; and
%	\item $  \mathrm{ad}^2_\nu (V_1) = V_3. $
%\end{enumerate}
	\begin{equation*}\label{eq:adjoint}
\mathrm{(i)}\quad \mathrm{ad}_\nu (V_1) = V_2 \ja \mathrm{(ii)}\quad\mathrm{ad}^2_\nu (V_1) = V_3.
\end{equation*}
	We prove first that, for every $ \nu \in \spann{Y_1,\dots,Y_n}\cup \R X $, the line $ t\mapsto \exp(t \nu) $ is abnormal.	Indeed, if $ \nu \in \R X $, then $ [\nu, V_2] = \{0\}$ by Lemma \ref{lemma:properties_of_engels}.ii. So $ \ad_\nu^2(V_1) \sus [\nu,V_2] = \{0\} $ and condition (ii) is not satisfied.
	Let then $ \nu \in \spann{Y_1,\dots, Y_n} $ and consider the linear map $ \ad_\nu \colon V_1 \to V_2 $. Notice that  $ \spann{Y_1,\dots, Y_n} \sus \ker( \ad_\nu) $, since $ \spann{Y_1,\dots, Y_n} $ is abelian by Lemma \ref{lemma:properties_of_engels}.i. Therefore, $ \ad_\nu(V_1) $ is 1-dimensional, which violates condition (i) as now $ \dim V_2 > 1 $.
	
	Assume then $ \nu \notin \spann{Y_1,\dots,Y_n}\cup \R X $ and let us show that the line $ t\mapsto \exp(t \nu) $ is non-abnormal by proving that $ \nu $ satisfies conditions (i) and (ii). Now  $ \nu $ can be written as
	\[
	\nu = aX + \sum_{i=1}^{n}a_i Y_i,
	\]
	where $ a,a_i \in \R $ are such that $ a\neq 0 $ and $ a_j\neq 0 $ for some $ j= 1,\dots,n $. Then 
	\begin{align*}
	[\nu, \spann{Y_1,\dots,Y_n}] &= \mathrm{span}{\{[\nu,Y_i] \,:\, i= 1,\dots,n\}} \\
	&= \mathrm{span}{\{[aX,Y_i] \,:\, i= 1,\dots,n\}} \\
	&=  \mathrm{span}{\{T_i \,:\, i= 1,\dots,n\}}\\
	&= V_2.
	\end{align*}
	Consequently, condition (i) is satisfied. To prove that also (ii) holds, let $ j\inn $ be such that  $ a_j\neq 0 $ and notice that $ [\nu, \R Y_j] = \R T_j $. Then
	\[
	\ad_\nu^2(\R Y_j) = [aX + \sum_{i=1}^{n}a_i Y_i, \R T_j] = [a_jY_j,\R T_j]= \R Z = V_3.
	\]
\end{proof}	
	\subsection{Proof of Theorem~\ref{t:nonCNP:char:step6}}
 We conclude with the following characterization of trimmed non-ample Carnot algebras in step 3. As a corollary, we shall obtain Theorem~\ref{t:nonCNP:char:step6}.
	\begin{prop}\label{prop:minimal_nonCNP_step3_isEngeln}
		Let $ \g $ be a
		stratified Lie algebra of step 3.
		If $ \g $ is  trimmed and not   ample,
		then  
		%and rank $ n+1 $. Then 
		it is isomorphic to some $ \mathbb{En}^n $. %, for some $n\in \N$.
	\end{prop}
	\begin{proof}
		We start by proving the following fact.
		
		If $ \g $ is a trimmed Lie algebra of step 3 and rank $ n+1 $ with a non-ample half-space $ W \sus V_1 $, then
		\begin{equation}\label{claim_abelian}
		\partial W  \text{ is abelian;}
		\end{equation}
		and
		\begin{equation}
		\label{claim_minimal}
		\begin{array}{c}
		%&&\nonumber 
		\text{for }   n+1\geq 3   \text{   every rank-} n\text{ Carnot subalgebra }     \h \text{ with } V_1(\h) \neq \partial W, \\
%		\text{the rank-}n\text{ Lie algebra }\lie{H} 
		\text{ is trimmed and not ample.} 
		% Lie algebra of rank $ n $ that is not    ample.  \\
		\end{array}
		\end{equation}
		%
		%	\\
		%	\\
		%	\emph{Claim 1.} $ \partial W $ is abelian.  \\
		%	\emph{Proof of Claim 1.} 
		Regarding the proof of \eqref{claim_abelian}, suppose by contradiction that there exist $ Y_1,Y_2 \in \partial W $ such that $ [Y_1,Y_2]\neq 0 $. 
Set $\s:= \Cl (\s_W )$. 
As $ \partial W \sus \edge   \s  $,  we have by Lemma~\ref{lem:CH}.2 that $  [Y_1,Y_2]\sus  \edge\s $. 
		Then on the one hand, by Lemma~\ref{lemma:step3_ideal(invV2)_is_inv} we have $ \Ideal_{\g}([Y_1,Y_2]) \sus \edge \s $.  On the other hand, since $ \g $ is trimmed, we have that $ V_3 \sus \Ideal_{\g}([Y_1,Y_2]) $, as the latter is a nontrivial ideal. 
		Hence, we infer that  $V_3 \sus\edge\s$. Consequently,   according to Lemma~\ref{lemma:step4_equiv_halfspace} we get a contradiction, since $W$ was assumed not to be  ample.   
		Thus \eqref{claim_abelian} is proved.
		%%	\\ \\	
		%%	\emph{Claim 2.} 
		
		Regarding the proof of \eqref{claim_minimal}, 
		we show first that
		\begin{equation}\label{primo_punto}
		 \Center(\h)\cap V_1(\h)= \{0\},
		\end{equation}
	\begin{equation}\label{secondo_punto}
	 \Center(\h)\cap V_2(\h)= \{0\}.
		\end{equation}
Before proving \eqref{primo_punto}, setting $H \coloneqq  V_1(\h) $,
		%	suppose that $ n+1\geq 3 $ and let $ H \sus V_1 $ be an $ n $-dimensional subspace of $ V_1 $ such that $ H \neq \partial W $. Then $\lie{H} $ is a minimal  Lie algebra of rank $ n $ that is not    ample.  \\
		%	\emph{Proof of Claim 2.} S
%		. \blue{ We show first that $    \Center(\h) \sus V_3(\h) $. 
		%This would prove that $ \h $ is a minimal step-3 Lie algebra, since $ V_3(\h)\sus V_3(\g) $ and $ V_3(\g) $ is one-dimensional by minimality of $ \g $.  As a first step, we prove that $   \Center(\h)\cap H= \{0\} $. 
		  we claim that
			\begin{equation}\label{terzo_punto} 
		\text{ if } 	X \in H\setminus \partial W \text{ and } Y\in \partial W \cap H \text{ satisfy }  [X,Y]= 0, \text{ then }  Y = 0 .
				\end{equation}
		  Indeed, since we have the decomposition $ V_1(\g) = \partial W \oplus \R X$
		  and $ \partial W $ is abelian by \eqref{claim_abelian}, we get that $Y$ commutes with $ V_1(\g)$. Since $ V_1(\g)$ generates $\g$, we get that $  Y\in \Center(\g)\cap V_1(\g)$, but the latter is trivial since $ \g $ is trimmed.
		  Hence, we get \eqref{terzo_punto}.
		  
Regarding the proof of 	\eqref{primo_punto}, we assume, aiming for contradiction, that there exists a nonzero element in  $   \Center(\h)\cap V_1(\h)= \Center(\h)\cap H$. We have two possibilities: either the element is of the form $X\in H\setminus \partial W $ or it is of the form  $Y\in \partial W \cap H$.
In the first case, since $  \partial W \cap H$ is  not trivial being the intersection of two hyperplanes, we can find a nonzero $Y$ in it contradicting  \eqref{terzo_punto}. In the second case, since $H\setminus \partial W $ is not empty being the two sets different hyperplanes, we can find $X$ in it contradicting  \eqref{terzo_punto}. Together these two cases prove \eqref{primo_punto}.
 
%		   so to express $ H $ as $ H = (\partial W \cap H) \oplus \R X $. Note that if $ Y\in \partial W \cap H $ satisfies $ [X,Y]=0 $, then $ Y \in    \Center(\g)\cap V_1(\g) $ since $ \partial W $ is abelian by \eqref{claim_abelian} and $ V_1(\g) = \partial W \oplus \R X$. But $    \Center(\g)\cap V_1(\g)=\{0\} $ as $ \g $ is minimal and hence $ [X,Y]\neq 0 $ for every nonzero $ Y \in \partial W\cap H $. Therefore $    \Center(\h)\cap H= \{0\} $. 
			
%			To prove that $  \Center(\h)\cap V_2(\h)= \{0\} $,
			Regarding \eqref{secondo_punto}, 
			 assume the contrary and let $ Z \in  \Center(\h)\cap V_2(\h) $ be nonzero. Applying \eqref{l:V2:cap:Z(g):univ:invariant} for $ \h $ gives   $   Z \in \pedge{\Cl( \s_{W\cap H})} $\,, where the latter is a subset of $ \edge{  \s }$. 
%			 As $ \bar \s_{W\cap H}\sus \bar \s_W $, we obtain that $ \R Z \sus \bar \s_W $. 
			 Similarly to the proof of \eqref{claim_abelian}, then by the fact that $ \g $ is trimmed and by Lemma~\ref{lemma:step3_ideal(invV2)_is_inv} we have
			\[
			V_3 \sus \Ideal_{\g}(Z) \sus\edge{  \s }.
			\]
			 Since $W$ is not ample. by Lemma~\ref{lemma:step4_equiv_halfspace} we get a contradiction. So  \eqref{secondo_punto} is proved.

			Properties \eqref{primo_punto} and  \eqref{secondo_punto}, together with the fact that $\Center(\h)$ is a non-trivial homogeneous subspace of $\h$, imply that
			$$\{0\} \neq  \Center(\h) \sus V_3(\h)\sus V_3(\g) .$$ Since $ V_3(\g) $ has dimension 1 as $ \g $ is trimmed, we get that $ \Center(\h)$ has dimension 1, i.e., we obtained that $ \h $ is a trimmed step-3 Lie algebra.  
			
%			Hence $ \Center(\h) \sus V_3(\h) $ and $ \h $ is a minimal step-3 Lie algebra.
			
			The fact that $ \h $ is not ample follows then from Lemma~\ref{lemma:CNPstep3_subalg_implies_CNP}: since $ \dim V_3(\g)  = 1$ and $ V_3(\h)\neq \{0\} $, then $ V_3(\h) = V_3(\g) $. Thus \eqref{claim_minimal} is proved.

		We complete the proof by an induction argument. The initial step is given for rank $ n + 1 = 2 $ by the classical Engel algebra $ \mathbb{En}^1 $, since every trimmed Carnot algebra of rank 2 and step 3 is isomorphic to $ \mathbb{En}^1 $ (see Remark \ref{rmk:filiforms_engel_quots}).  Let then $ n+1\geq 3 $ and suppose, by induction assumption, that every trimmed step-3  non-ample Lie algebra of rank $ n $ is isomorphic to $ \mathbb{En}^{n-1} $. Let $ \g $ be a trimmed step-3  non-ample Lie algebra of rank $ n+1 $. By \eqref{claim_abelian} and \eqref{claim_minimal}, the Lie algebra $ \g $ has a unique abelian  stratified subalgebra of rank $ n $ and, by induction assumption, every other rank-$ n $ stratified subalgebra is isomorphic to $ \mathbb{En}^{n-1} $. As $ \dim V_3(\g) = 1 $ since $ \g $ is trimmed, we conclude that $ \g $ is isomorphic to $ \mathbb{En}^{n} $ by Proposition~\ref{prop:char_of_engels}.
	\end{proof}

  	\subsubsection{Proof of Theorem~\ref{t:nonCNP:char:step6}} To prove the easy direction of Theorem~\ref{t:nonCNP:char:step6}, we recall that every Engel-type algebra is not ample, see Proposition~\ref{prop:Engels:non-CNP}. Hence, if a Carnot algebra $\g$ has an Engel-type algebra as a quotient, then $\g$ is not ample, see Proposition~\ref{prop:quotients_product_ample}.
	Regarding the other implication of Theorem~\ref{t:nonCNP:char:step6}, let $\g$ be a step-3 Carnot algebra that is not   ample.
 By Proposition~\ref{prop:ample_quotient_small_center}, there exists a quotient algebra of $ \g $ that is
		trimmed and not     ample. Notice that, being not ample, still such a quotient has step 3. 
 Proposition~\ref{prop:minimal_nonCNP_step3_isEngeln} makes us conclude that such a quotient is isomorphic to some 
 Engel-type algebra. \qed

%%%%%%%%%%%%%%%%%%%%%%%%%%%%%%%%%%%%%%%%%%%%%%%%

   \bibliography{general_bibliography}
\bibliographystyle{amsalpha}

\end{document}